\documentclass[12pt]{amsart}

\usepackage[utf8]{inputenc}
\usepackage{xcolor}
\definecolor{codexgreen}{RGB}{0,105,92}
\newif\ifcodexmarks
\codexmarkstrue

\usepackage{amsmath} % Símbolos refinados de matemáticas
\usepackage{graphicx} % Incluir imágeness
\usepackage{subcaption} % Poder usar subfiguras
\usepackage{amsfonts}
\usepackage{mathrsfs}
\usepackage{amssymb} % Los símbolos el los conjuntos numéricos
\usepackage{amsthm}   % <--- ESTE es el que crea el entorno 'proof'

\usepackage{amssymb}
\usepackage[all]{xy}
\usepackage{tikz-cd}
\newcommand{\A}{\mathscr A}

\newtheorem*{theorem*}{Theorem}
\newtheorem{theorem}{Theorem}[section]
\newtheorem{lemma}[theorem]{Lemma}
\newtheorem{corollary}[theorem]{Corollary}
\newtheorem{proposition}[theorem]{Proposition}

\newcommand{\To}{\longrightarrow}

\theoremstyle{definition}
\newtheorem{defn}[theorem]{Definition}
\newtheorem{example}[theorem]{Example}
\newcommand{\adef}{\begin{defn}}
\newcommand{\zdef}{\end{defn}}
\def\PB{\operatorname{PB}}

\def\Ker{\operatorname{Ker}}
\def\supp{\operatorname{supp}}

\newtheorem*{claim*}{Claim}
\newtheorem*{remark*}{Remark}

\title{The Serre closure of Banach spaces}

\dedicatory{Dedicated to Professor Serre on his 100-th birthday, with affection and gratitude}

\date{\today}

\begin{document}

\author[Cabello]{F\'elix Cabello S\'anchez}
\address{Universidad de Extremadura, Instituto de Matem\'aticas Imuex\\
Avenida de Elvas s/n\\ 06011-Badajoz\\ Spain} \email{fcabello@unex.es}

\email{castillo@unex.es}

\email{raulmb@unex.es}

\email{navarrogarmendia@unex.es}

\author[Castillo]{Jes\'us M.\,F. Castillo}
%\address{Universidad de Extremadura, Instituto de Matem\'aticas Imuex\\
%Avenida de Elvas s/n\\ 06011-Badajoz\\ Spain} \email{castillo@unex.es}
% CODEX NOTE: this duplicates the first author's email address; confirm the
% intended address for the second author before submission.

\author[Mart\'inez]{\\ Ra\'ul Mart\'inez-Boh\'orquez}
%\address{Universidad de Extremadura, Instituto de Matem\'aticas Imuex\\
%Avenida de Elvas s/n\\ 06011-Badajoz\\ Spain}
%\email{raulmb@unex.es}

\author[Navarro]{Jos\'e Navarro Garmendia}
%\address{Universidad de Extremadura, Instituto de Matem\'aticas Imuex\\
%Avenida de Elvas s/n\\ 06011-Badajoz\\ Spain} \email{navarrogarmendia@unex.es}

\thanks{This research has been supported in part by Projects PID2022-142024NB-I00 (SIGA) and PID2023-146505NB-C21 funded by MCIU/AEI/10.13039/5011000110331/FEDER/UE as well as Projects IB24002 and GR24068, funded by Junta de Extremadura and the EU-ERDF}

\begin{abstract}
We determine the Serre closure of Banach spaces in the category of quasi-Banach spaces, that is, the smallest class of quasi-Banach spaces that contains all Banach spaces and is stable under subspaces, quotients, and extensions. Precisely, a quasi-Banach
space $X$ belongs to this closure if and only if there exist constants $M$ and $k\in\mathbb N$ such that
$$
\frac{\|x_1+\dots+x_n\|}{\|x_1\|+\dots+\|x_n\|}\leq M\log^k n
$$
for every $n\geq 2$ and every $x_1,\dots,x_n\in X$; equivalently, $X$ is a
quotient of an Orlicz sequence space $\ell(I,f_l)$, where $I$ is an
index set and $f_l(t)$ is equivalent to $t\log^l(1/t)$ near zero.

The proof
combines geometrical estimates for quasi-Banach spaces and quasilinear maps with the study of higher order derived spaces associated to complex interpolation between $\ell_{2/3}$ and $\ell_2$. We provide examples of spaces in the Serre closure of Banach spaces that are not iterated extensions of Banach spaces and discuss some related problems.
\end{abstract}

\maketitle

\section{Introduction} The reader is addressed to the forthcoming sections for any unexplained notation.

In an additive category $\mathbf{C}$, a class of objects  is called a Serre class if it is closed under subobjects, quotients, and extensions. Equivalently, for any short exact sequence $A \to B \to C$ in $\mathbf{C}$, the middle object $B$ belongs to the class if and only if both $A$ and $C$ do.

If $\bf D$ is a class of objects of $\bf C$, the Serre closure of $\bf D$ in $\bf C$, denoted by  $\langle\mathbf{D}\rangle_\mathbf{C}$, is the smallest Serre class containing $\mathbf{D}$. Although many interesting classes of Banach spaces are Serre (many renowned specimens can be seen stuffed in \cite{3sp}), Banach spaces themselves do not form a Serre class within the category of quasi-Banach spaces and operators, which is the natural framework for studying extensions in this setting. Indeed, in the late seventies, Roberts, Kalton, and Ribe independently constructed nontrivial extensions of $\ell_1$
(the Banach space of all summable sequences) by the line, the last of which has achieved the greatest fame.

In this paper, we compute the Serre closure of Banach spaces in the category of quasi-Banach spaces and (linear, bounded) operators.
We denote by $\bf Q$ and $\bf B$ the category of quasi-Banach spaces and the class of Banach spaces, respectively, and we shamelessly write $X\in\mathbf B$ and the like.

Given a quasi-Banach space $X$, we consider the sequence
%Our characterization  introduces the sequence
\begin{equation*}
c(n;X) = \sup \frac{\|x_1+\dots+x_n\|}{\|x_1\|+\dots+\|x_n\|},
\end{equation*}
%of a quasi-Banach space $X$,
where the sup is taken over all families $x_1, \dots, x_n$ in $X$, to obtain:

\begin{theorem*}\label{THEOREM} For a quasi-Banach space $X$ the following are equivalent:
\begin{enumerate}
\item $X$ belongs to the Serre closure of the class of Banach spaces in $\bf Q$.
\item There exist $k\geq 0$ and $M>0$ such that
$c(n;X) \leq M \log^k n$ for every $n\geq2$.
\item There exist an index set $I$ and an integer $l\geq0$ such that $X$ is a
quotient of the sequence space $\ell(I,f_l)$.
\end{enumerate}
\end{theorem*}

The spaces that appear in (3) are the Orlicz ``sequence'' spaces associated with the functions $t\mapsto \log^l(1/t)$ for $t>0$ near zero.
Thus, our main result not only describes the Serre closure of Banach spaces, but also provides the fundamental examples. The fact that something as simple as the behaviour of the sequence $c(n;X)$
characterizes the spaces in $\langle\mathbf{B}\rangle_\mathbf{Q}$
might be seen as a stroke of luck, but after all, an arbitrary quasi-Banach space carries no structure beyond addition and the quasinorm,
so what else could we possibly look at? The remainder of the paper is devoted to explaining the proof of the theorem and some concepts involved. The proof that a quasi-Banach space in the Serre closure of Banach spaces must satisfy the upper estimate (2) is mostly boring calculations with quasilinear maps, culminating in Proposition \ref{1implies2}. The proof that the upper estimate (2) forces $X$ to be a quotient of some Orlicz space is contained in the not-so-straightforward calculations leading to Proposition \ref{2implies3}. Finally, that the Orlicz spaces in (3) belong to $\langle\mathbf{B}\rangle_\mathbf{Q}$ relies critically on complex interpolation theory for quasi-Banach spaces plus the study of the higher  ``derived''  spaces developed in Section \ref{3implies1}.\smallskip

In order to make the paper self-contained, Section
\ref{serreclosure} contains the quasi-Banach-space versions of the categorical
facts we need, Section~\ref{subsec:quasilinear} contains the necessary
background on quasilinear maps and twisted sums, and Section~\ref{3implies1}
presents the complex-interpolation facts used in the proof.\medskip

The obvious ancestor of our work is \cite{kalton81}, one of Kalton's many other masterpieces. Kalton focuses on the first level of the Serre closure, where his results are unbeatable: he proves that a quasi-Banach space $X$ is a quotient of a subspace of an extension of two Banach spaces if and only if it is a quotient of $\ell(f_1)$ and that this happens if and only if $X$ is log-convex, which means that
$$
\|x_1+\dots+x_n\|\leq M\sum_{i\leq n}\|x_i\|(1+\log i)
$$
for some constant $M$, every $n\in\mathbb N$ and every
$x_1, \dots, x_n\in X$; see Section~3 (especially Theorems 3.3 and 3.6) and Theorem 7.2 of the aforementioned paper.

Finally, if the reader is wondering what AI has ever done for us, see at Lemma \ref{lem:compactness}.

\section{The Serre closure in a quasi-Abelian category}\label{serreclosure}

The reader can skip this section if he got past the title without a sudden fright. A category $\bf C$ is called \emph{additive} if:
\begin{itemize}
    \item Each set $\operatorname{Hom}(A, B)$ is enriched with a structure of Abelian group in such a way that composition distributes over addition.
\item There exists a zero object (one which is both initial and terminal) in $\bf C$.
\item $\bf C$ has finite products (which are automatically direct sums).
\end{itemize}

An additive category that has all kernels and cokernels (hence, all finite limits, in particular pullbacks and pushouts) is called \emph{pre-Abelian}. A pre-Abelian category is
\emph{quasi-Abelian} if it satisfies the following mutually dual conditions concerning commutative squares
$$
\xymatrix{A \ar[r] & B \\
C \ar[r] \ar[u] & D \ar[u]}
$$
\begin{itemize}
    \item If the square is a pushout diagram and $A\to B$ is a kernel, then so is $C\to D$.
    \item If the square is a pullback diagram and $A\to B$ is a cokernel, then so is $C\to D$.
\end{itemize}

\subsection{Serre classes}\label{sec:serre classes}
The notion of a short exact sequence can be succinctly defined in any additive category as a kernel-cokernel pair. That is, a sequence of morphisms $A \to B \to C$ is a short exact sequence (we omit the initial and final zeros) if $A \to B$ is the kernel of $B \to C$, and $B \to C$ is the cokernel of $A \to B$. Let $\bf C$ be a quasi-Abelian category and let $\bf S$ be a class of objects of $\bf C$ that we regard as a full subcategory. As we said in the introduction, $\bf S$ is a Serre class if, whenever $A \to B \to C$ is a short exact sequence in $\bf C$, the middle object $B$ is in $\bf S$ if and only if $A$ and $C$ are.\medskip

Regarding the meaning of an \emph{extension}, let us say that $A$ is a subobject of $B$ if there is a morphism $A \to B$ which is a kernel, and that $C$ is a quotient of $B$ if there is a morphism $B \to C$ which is a cokernel. We say that $B$ is an extension of $C$ by $A$ if it fits into a short exact sequence $A \to B \to C$. The fact that in homological algebra the term ``extension'' usually refers to the exact sequence itself, rather than to the middle object, will not cause any confusion, we hope.

Thus, $\bf S$ is a Serre class if (and only if) it is closed under subobjects, quotients and extensions.
In general one can define $\langle{\bf S}\rangle_{\bf C}$, the Serre closure of $\bf S$ in $\bf C$ as the smallest Serre subclass of $\bf C$ that contains $\bf S$. If there is no risk of confusion, we will omit $\bf C$.

To illustrate how this works and to fix our notation, let us consider some examples arising from Functional Analysis:

\begin{itemize}
    \item The category $\bf TVS$ of (Hausdorff) topological vector spaces and continuous operators is quasi-Abelian and so is $\bf F$, the full subcategory of complete, metrizable linear spaces.
    \item The class $\bf Q$ of quasi-Banach spaces (equivalently, complete, metrizable, locally bounded spaces) is a Serre subclass of $\bf F$.
    \item The class $\bf B$ of Banach spaces fails to be Serre in $\bf Q$.
    \item The class $\bf H$ of Hilbert spaces fails to be Serre in $\bf B$.
\end{itemize}

Other Serre and non-Serre classes will appear liberally sprinkled throughout the paper.

Serre classes are mostly studied in Abelian categories. In fact, Serre’s ``foundational'' paper \cite{serre} deals with Abelian groups, the most Abelian of all categories, but see also \cite{kanda}. However, since our study takes place in a quasi-Abelian category for obvious reasons, we present some direct consequences of the definitions in this setting.
Before we proceed, let us define three operations on classes. Let $\mathbf{C}$ be a quasi-Abelian category and $\mathbf{D}, \bf E$ be classes of objects of $\bf C$, which we view as full subcategories:

\begin{itemize}
    \item $\mathscr{S}(\mathbf{D})$ is the class of subobjects of objects in $\mathbf{D}$.
    \item $\mathscr{Q}(\mathbf{D})$ is the class of quotients of objects in $\mathbf{D}$.
    \item $\mathscr{E}(\mathbf{D}, \mathbf{E})$ consists of those $C\in\mathbf C$ that fit into a short exact sequence of the form $E \to C \to D$, where $D \in \mathbf{D}$ and $E \in \mathbf{E}$. If $\mathbf{D} = \mathbf{E}$, we simply write $\mathscr{E}(\mathbf{D})$.
\end{itemize}

\begin{lemma}
    Let $\mathbf{D}, \bf E$ and $\bf G$ be classes of objects of a quasi-Abelian category $\bf C$.
    \begin{itemize}
        \item[{(a)}] $\mathscr{SQ}(\bf D)=\mathscr{QS}(\bf D)$.
     \item[{(b)}] $\mathscr E(\mathscr E({\bf D}, {\bf E}), {\bf G})= \mathscr E({\bf D}, \mathscr E({\bf E}, {\bf G}))$.
    \end{itemize}
\end{lemma}

\begin{proof}
(a) Assume that $S$ is a subobject of $Q$, which in turn is a quotient of $D$, an object of $\bf D$. Then the pullback diagram
$$
\xymatrix{D \ar[r] & Q \\
\PB \ar[r] \ar[u] & S \ar[u]}
$$
and the very definition of quasi-Abelian category yield that $S$ is also a quotient of a subobject of $D$. The other inclusion follows by (categorical) duality, using pushouts.

(b) Assume that $H$ belongs to $\mathscr E(\mathscr E({\bf D}, {\bf E}), {\bf G})$ and let $G \to H \to F$ and $E \to F \to D$ be short exact sequences witnessing it. Observe the diagram
$$
\xymatrix{& D \ar@{=}[r] & D \\
G \ar[r] \ar@{=}[d] & H \ar[r] \ar[u] & F \ar[u] \\
G \ar[r] & \PB \ar[r] \ar[u] & E \ar[u]}
$$
until it becomes clear that $H\in \mathscr E({\bf D}, \mathscr E({\bf E}, {\bf G}))$. The other containment follows by duality, just using pushouts --- actually the diagram is \emph{also} a pushout diagram!
\end{proof}

The preceding lemma is the quasi-Abelian counterpart of the first two items of \cite[Proposition~2.4]{kanda}. The third part also works: define ${\bf D}^{(n)}$ taking ${\bf D}^{(0)}$ as the zero objects and then ${\bf D}^{(n+1)}= \mathscr E({\bf D}, {\bf D}^{(n)})$. Then $\bigcup_{n\geq 0} {\bf D}^{(n)}$ is the smallest class that is closed by extensions and contains $\bf D$ just because one has  ${\bf D}^{(n+k)}= \mathscr E({\bf D}^{(n)}, {\bf D}^{(k)})$ for all $n,k\geq 0$. We shall show later, by means of counterexamples, that the remainder of \cite[Proposition~2.4]{kanda} fails in quasi-Abelian categories.
Thus, we must reach the Serre closure in another way. Assume that $\bf D$ is a class of objects of $\bf C$ we define $({\mathscr{S\hspace{-3pt}QE}})^{n}\bf D$ (this time the superscript does not take parentheses!) letting $({\mathscr{S\hspace{-3pt}QE}})^{0}(\bf D)=\bf D$ and then
$$
({\mathscr{S\hspace{-3pt}QE}})^{n+1}{\bf D}=
{\mathscr{S\hspace{-3pt}QE}}({\mathscr{S\hspace{-3pt}QE}})^{n}{\bf D}=
(\mathscr{S}(\mathscr{Q}(\mathscr{E}({\mathscr{S\hspace{-3pt}QE}})^{n}{\bf D})))).
$$
Note that, according to our definitions, $({\mathscr{S\hspace{-3pt}QE}})^{n}{\bf D}$ contains ${\bf D}^{(2^{n})}$. The following result should be obvious % easy
by now:

\begin{proposition} If $\bf C$ is quasi-Abelian, then $\langle {\bf D}\rangle_{\bf C} = \bigcup_{n\geq 0} ({\mathscr{S\hspace{-3pt}QE}})^{n}\bf D$.
\end{proposition}

\subsection{Quasi-Banach spaces}
We will now provide a few more definitions here, just in case the reader does not know or does not remember them.
General references on quasi-Banach spaces are \cite{kpr, K-quasi} and \cite{HMBST}, the latter being more categorically oriented. % which is more categorically oriented.

Let $X$ be a linear space over $\mathbb K$, which can be either $\mathbb R$ or $\mathbb C$. A quasinorm on $X$ is a function $\|\cdot\|: X\to [0,\infty)$
that vanishes only at the origin, is ``positively homogeneous'' in the sense that  $\|\lambda x\|=|\lambda|\|x\|$ for every $x\in X$ and  every $\lambda\in\mathbb K$ and satisfies the following weak version of the triangle inequality: there is a constant $c\geq 1$ such that
\begin{equation*}
\|x+y\|\leq c(\|x\|+\|y\|)
\end{equation*}
for every $x,y\in X$. The smallest $c$ satisfying the preceding inequality  is often called the ``modulus of concavity'' of the quasinorm. Each quasinorm induces a linear topology on $X$ characterized by the fact that the unit ball $B=\{x\in X: \|x\|\leq 1\}$ is a neighbourhood of the origin. If $X$ is complete with respect to the associated uniform structure, we call it a quasi-Banach space.

A $p$-norm, $p\in(0,1]$, is a quasinorm that additionally satisfies $\|x+y\|^p\leq \|x\|^p+\|y\|^p$. The Aoki--Rolewicz theorem states that each quasinorm is equivalent to a $p$-norm for some $p\in(0,1]$.
The meaning of $p$-normed and $p$-Banach space should be obvious. If $p=1$ we speak of norms, normed spaces and Banach spaces.\medskip

Throughout the paper, we will use the following conventions: if $f$ and $g$ are functions defined on the same set $S$ we write $f(s)\lesssim g(s)$ to mean that there exists a constant $M>0$ such that $f(s)\leq Mg(s)$ for every $s\in S$, and $f(s)\sim g(s)$ means that
$f(s)\lesssim g(s)$ and $g(s)\lesssim f(s)$. Most of the time our formulae will contain more than one variable, or parameters, in which case we write, as before, $f(s, t)\lesssim g(s, t)$ to mean that there exists a constant $M>0$ such that $f(s,t)\leq Mg(s,t)$ for every $(s,t)\in S\times T$, while $f(s, t)\lesssim g(s, t)$ for fixed $t$ means that for every $t\in T$ there is $M$, possibly depending on $t$, such that $f(s,t)\leq Mg(s,t)$ for every $s\in S$.

\section{Sizes of sums in quasi-Banach spaces --- Proof of $(1)\Longrightarrow
(2)$}

From now on, we focus exclusively on the Serre closure of Banach spaces  within the category of quasi-Banach spaces. In this context, kernels are always called embeddings and subobjects are called subspaces. We will keep this convention.
Admittedly, in this context extensions are often called twisted sums, a convention that we shall not follow. On the other hand, the open mapping theorem guarantees that a sequence of quasi-Banach spaces and operators is exact in $\bf Q$ if and only if it is exact in the category of vector spaces.\medskip

%Our study is based on the asymptotic behaviour of sums. %Given two sequences of real numbers $(f(n))_n$ and $(g(n))_n$ we will write
%$$f\lesssim g$$
%to mean that $f(n)\leq C g(n)$ for some constant $C>0$ and all $n$. I.e., $f=O(g)$, with the bonus that we will often write the more descriptive
%$f(n)\lesssim g(n)$ (that we cannot write with $O(\cdot)$). Given two functions, $f,g$, we will write $f\sim g$ to mean that there are constants $a,b$ such that $af(x)\leq g(x)\leq bf(x)$.

Recall from the introduction that, given a quasi-Banach space $X$ and $n\geq 2$, we denote by $c(n;X)$ the least number $c$ such that
$$
\|x_1+\dots+x_n\|\leq c \left( \|x_1\|+\dots+\|x_n\|  \right).
$$
whenever $x_1,\dots,x_n\in X$. Obviously $c(1)$ would be $1$, but we prefer to exclude this case. When the space $X$ is clear from the context, we will simply write $c(n)$. Some nearly trivial remarks on $c(n;X)$ follow:
\begin{itemize}
    \item $c(2; X)$ is the modulus of concavity of $X$.
    \item $c(n;X)$ is increasing and submultiplicative: $c(nk; X)\leq c(n;X)c(k;X)$.
    \item $\sup_n c(n;X)<\infty$ if and only if $X$ is isomorphic to a Banach space.
    \item If $Y$ is is isomorphic to either a subspace or a quotient of $X$, then $c(n;Y)\lesssim c(n;X)$.
\end{itemize}
The subspace part is immediate. The part about quotients is as follows: If $\pi:X\to Y$ is a quotient map, then there exists $M$ such that for every $y\in Y$ one can pick $x\in X$ such that $\pi x=y$, with $\|x\|\leq M\|y\|$. Now, if $y_1,\dots,y_n\in Y$, take $x_1,\dots,x_n\in X$ such that $\pi x_i=y_i$ with $\|x_i\|\leq M\|y_i\|$. Then
$$
\left\|\sum_{i\leq n} y_i\right\|=
\left\|\pi \left(\sum_{i\leq n} x_i\right)\right\|\leq \|\pi\| c(n;X) \sum_{i\leq n} \|x_i\|\leq M\|\pi\|c(n;X) \sum_{i\leq n} \|y_i\|.
$$

\subsection{Quasilinear maps}\label{subsec:quasilinear}
We have just seen that the sequence of constants $c(n)$
cannot grow when taking subspaces or quotients. To study its behaviour under the formation of extensions (the other procedure that must be controlled), we need a manageable description of them. The best approach we have found is to use quasilinear functions, which we now present without hesitation. \medskip

Let $X$ and $Y$ be quasinormed spaces and let $W$ be a linear space containing $Y$ in the purely algebraic sense (actually %as we shall see very soon
one may assume $W=Y$, but some flexibility is convenient here). A homogeneous mapping $F:X\to W$ (not $Y$) is said to be quasilinear from $X$ to $Y$ if:
\begin{itemize}
    \item For every $x,x'\in X$ one has $F(x+x')-F(x)-F(x')\in Y$.
    \item There is a constant $d$ such that $\|F(x+x')-F(x)-F(x')\|\leq d(\|x\|+\|x'\|)$ for all $x,x'\in X$.
\end{itemize}
The smallest possible $d$ for which the preceding inequality holds is traditionally called the quasilinearity constant of $F$, denoted by $Q(F)$. Given such a quasilinear map $F$, we can construct an extension of $X$ by $Y$ as follows. We consider the linear space
$
Y\oplus_F X =\{(w,x)\in W\times X: w-F(x)\in Y\}
$
 endowed with the quasinorm
$$
\|(w,x)\|_F = \|w-F(x)\|+\|x\|,
$$
which is called, with good reason, the quasinorm induced by $F$. One then has an exact sequence
$$
\xymatrix{0\ar[r]&Y \ar[r]^-\imath & Y \oplus_F X \ar[r]^-\pi & X\ar[r]&0}
$$
in which $\imath(y)=(y,0)$ and $\pi(w,x)=x$. Moreover, $\imath$ preserves the quasinorm, and $\pi$ maps the unit ball of $Y\oplus_F X$
onto that of $X$. The space $Y\oplus_F X$ is complete when $X$ and $Y$ are (see \cite[Proposition 2.3.4]{HMBST}) and, by \cite[Proposition 3.3.7]{HMBST}:

\begin{lemma}\label{lem:Z=YFX} Every extension of $X$ by $Y$ is isomorphic to one of the form $Y\oplus_F X$ for some quasilinear map $F:X\to Y$. Moreover, if $F:X\to Y$ is quasilinear then for every scalar $a\neq 0$ the map $Y\oplus_{aF} X\to Y\oplus_F X$ sending $(y,x)$ to $(y,ax)$ is an isomorphism. \end{lemma}

We now introduce another sequence associated with $F:X\to Y$ that will allow us to control $c(n; Y\oplus_F X)$.

\adef Given a quasilinear map $F:X\to Y$ and $n\geq 2$ we denote by $d(n;F)$
the least constant $d$ such that
$$
\|F(x_1+\dots+x_n)- Fx_1- \dots-Fx_n\|\leq d \left( \|x_1\|+\dots+\|x_n\|  \right)
$$
whenever $x_i\in X$ for $i\leq n$.% are vectors of  $X$.
\zdef

If the quasilinear map $F$ is clear from the context, we will just write $d(n)$. The sequence $d(n;F)$ is monotone increasing and  $d(2;F)=Q(F)$.

\begin{lemma}\label{lem:c and d} Let $F:X\to Y$ be a quasilinear map. Then:
\begin{itemize}
    \item[(a)] $c(n;Y\oplus_F X)\leq
\max\Big\{  c(2;Y)c(n;Y), c(2;Y)d(n;F)+ c(n;X) \Big\}$
    \item[(b)] $
d(2n; F)\leq c(3;Y)\Big( d(2;F) c(n;X) + d(n;F) \Big)
$ .
\item[(c)] If $d(2;F)\leq 1$, then for every $j$ one has $$
    d(2^j; F) \leq c(j; Y) \sum_{0\leq l <j} c(2^l; Y) c(2^{j-l-1}; X).$$
\end{itemize}
\end{lemma}

\begin{proof}
We will sometimes write $\Delta F(x_1,\dots,x_n)= F(\sum_{i\leq n}x_i)- \sum_{i\leq n}Fx_i$ to save space.

 To prove (a), pick $n$ points $(y_i,x_i)$ in $Y\oplus_F X$. Put $y=\sum y_i$ and $x=\sum x_i$, so that $(y,x)=\sum (y_i,x_i)$. We have
\begin{align*}
    \|(y,x)\|_F &= \|y-Fx\|+\|x\| =   \|\sum y_i-\sum Fx_i + \sum Fx_i- Fx\|+\|x\|\\
    &\leq c(2;Y) \left( \|\sum (y_i-Fx_i)\| + \|\Delta F(x_1,\dots,x_n)\|  \right) + c(n;X)\sum \|x_i\|\\
    & \leq c(2;Y) \left( c(n;Y)  \sum\|y_i-Fx_i\| + d(n;F)\sum \|x_i\| \right) + c(n;X)\sum \|x_i\|,
\end{align*}
which is enough. To prove (b), % we may assume $d(2;F)\leq 1$ --- otherwise just divide $F$ by $d(2;F)$.
% BEGIN CODEX CHANGE C40 (local corrections in Proposition 3.1)
 pick $2n$ points in $X$ and label them as $x_1,\dots,x_n; x_1', \dots, x_n'$. Put $x=\sum_{i\leq n}x_i$ and $x'=\sum_{i\leq n}x_i'$. We have
\begin{align*}
&\|\Delta F(x_1, \dots,x_{{n}}, x_1',\dots,x'_{{n}})\|\\
&= \|F(x+x')- F(x)- F(x') + F(x)- \sum Fx_i +F(x')-\sum Fx_i'\|\\
&\leq c(3;Y)\left(\|F(x+x')- Fx-
Fx'\| + \left\|Fx- \sum Fx_i\right\| +\left\|Fx'-\sum Fx_i'\right\|\right)\\
&\leq c(3;Y)\left(d(2;F)(\|x\|+ \|x'\|)+ d(n;F)\sum\|x_i\|+ d(n;F)\sum\|x_i'\| \right)\\
&\leq c(3;Y)\left(d(2;F)c(n;X)\left(\sum\|x_i\|+ \sum\|x_i'\|\right) + d(n;F)\left(\sum\|x_i\|+ \sum\|x_i'\|\right) \right).
\end{align*}

The proof of Part (c) is quite easy.
However, it requires some effort to formalize the process of dividing the $2^j$
vectors from the hypothesis into successive halves until each group consists of a single element. We visualize this process using a dyadic tree of depth $j$ as follows.
$$
\xymatrixcolsep{-0.2pc}
\xymatrix{
&&&&& &&&&&&&&& x_{\varnothing} \ar@{-}[dllllllll]  \ar@{-}[drrrrrrr] \\
&&&&&x_0  \ar@{-}[dlll]  \ar@{-}[drrrr] &&  &&&&&&&&&&&&&&&x_1 \ar@{-}[dllll]  \ar@{-}[drrr]\\
&&x_{00} \ar@{-}[dll]  \ar@{-}[drr] &&&&&&&& x_{01}  \ar@{-}[dll]  \ar@{-}[drr] &&&&&&&&x_{10}  \ar@{-}[dll]  \ar@{-}[drr] &&&&&&&x_{11}  \ar@{-}[dl]  \ar@{-}[drrr] &&&\\
x_{000}&&&& x_{001}&&&&x_{010}&&&&x_{011}&&&&x_{100}&&&& x_{101}&&&&x_{110}&&&&x_{111}&
}
$$

Given an integer $\ell\geq 0$ we regard ${\bf 2}^l= \{0,1\}^l$ as the set of words $a_1\dots a_l$ of length $l$, with $a_k\in\{0,1\}$. If $l=0$ we only have the ``empty word'' $\varnothing$.  Each word $a=a_1\dots a_l$ in ${\bf 2}^l$ has two immediate successors in ${\bf 2}^{l+1}$, namely, $a0=a_1\dots a_l0$ and $a1=a_1\dots a_l1$;  and $2^k$ descendants in  ${\bf 2}^{l+k}$. The dyadic tree of depth $j$ is just $T=\bigcup_{0\leq l\leq j} {\bf 2}^{l}$. Now pick $2^j$ vectors of $X$ indexed by $a\in {\bf 2}^{j}$ and assume without any loss of generality that
 $\sum_{a\in {\bf 2}^{j}}\|x_a\|=1$. We define a tree-map $T\to X$ recursively as follows:
 \begin{itemize}
     \item For $a\in {\bf 2}^{j}$ the value of the map at $a$ is just $x_a$.
     \item For $a\in {\bf 2}^{l}$ with $0\leq l<j$ put $x_a=x_{a0}+ x_{a1}$, the sum of the values of the immediate successors.
 \end{itemize}
Note that each $x_b\in {\bf 2}^{l} $ equals the sum of those $x_a$ whose indices are descendents of  $b$ in ${\bf 2}^{j}$. In particular, $x_\varnothing = \sum_{a\in {\bf 2}^j} x_a$. Moreover
$$
F(x_\varnothing)-\sum_{i\in{\bf 2}^j}F(x_i) =
\sum_{0\leq l<j} \sum_{a\in{\bf 2}^l} \Delta F(x_{a0}, x_{a1}).
$$
Hence, taking into account that $d(2;F)\leq 1$,
\begin{align*}
\left\| F(x_\varnothing)-\sum_{i\in{\bf 2}^j}Fx_i\right\| &\leq c(j;Y)  \sum_{0\leq l<j} \left\|  \sum_{a\in{\bf 2}^l} \Delta F(x_{a0}, x_{a1})\right\| \\
&\leq c(j;Y)  \sum_{0\leq l<j} c(2^l;Y)   \sum_{a\in{\bf 2}^l} \| \Delta F(x_{a0}, x_{a1})\| \\
&\leq c(j;Y)  \sum_{0\leq l<j} c(2^l;Y)   \sum_{a\in{\bf 2}^l} \left( \| x_{a0}\| + \|x_{a1}\| \right). \\
%&\leq c(j;Y)  \sum_{0\leq l<j} c(2^l;Y)   \sum_{a\in{\bf 2}^l} \left( \| x_{a0}\| + \|x_{a1}\| \right)
\end{align*}
But, if $a\in {\bf 2}^l$, then $a0\in {\bf 2}^{l+1}$ and $x_{a0}$ is the sum of the $2^{j-l-1}$ ``initial'' vectors $x_b$ whose indices begin with $a0$, and similarly with $x_{a1}$, so that
$$
 \| x_{a0}\| + \|x_{a1}\|\leq c(2^{j-l-1}; X)\sum_{b} \|x_b\|,
$$
where the sum is taken over the descendents of $a$ in ${\bf 2}^j$.
Therefore,
$$
\sum_{a\in{\bf 2}^l} \left( \| x_{a0}\| + \|x_{a1}\| \right)
\leq c(2^{j-l-1}; X)\sum_{b\in{\bf 2}^j} \|x_b\| = c(2^{j-l-1}; X),
$$
and the result follows.
\end{proof}

We are now ready to estimate the sum constants of iterated extensions of Banach spaces and their relatives. The following result provides manageable upper estimates, the first of which will eventually prove to be optimal.

\begin{proposition}\label{1implies2} Let $Z$ be a quasi-Banach space.
    \begin{itemize}
        \item[(a)]  If $Z$ is in ${\bf B}^{(k)}$ then $c(n;Z)\lesssim \log^{k-1} n$.
         \item[(b)] If $Z$ is in $(\mathscr{S\hspace{-3pt}QE})^k\bf B$ then $c(n;Z)\lesssim\log^{2^k-1+\varepsilon}(n)$ for every $\varepsilon>0$
    \end{itemize}
\end{proposition}

\begin{proof} Observe that if $a$ and $b$ are positive, increasing sequences such that $a(2^n)\lesssim b(2^n)$ and $a$ satisfies the $\Delta_2$-condition (that is: $a(2n)\leq Ca(n)$ for some $C$ and all $n$), then $a(n)\lesssim b(n)$. Indeed, assuming $a(2^n)\leq M b(2^n)$ for all $n$, pick $m\in(2^n, 2^{n+1})$. Then
$$
a(m)\leq a(2^{n+1})\leq Ca(2^n)\leq MCb(2^n)\leq MCb(m).\\
$$

Now, to prove (a) we proceed by induction on $k\geq 1$. If $k=1$ then $Z$ is (isomorphic to) a Banach space. If $Z$ carries a norm, then   $c(n;Z)=1$ for all $n$; if it had an equivalent quasinorm, the sequence
 $c(n;Z)$ would be bounded. The induction step is as follows: assume $k>1$ and that $Z$ is a quasi-Banach space in ${\bf B}^{(k)}$. Take a short exact sequence $Y \to Z \to X$ with $X$ in ${\bf B}^{(k-1)}$ and $Y$ Banach. By Lemma~\ref{lem:Z=YFX} we may safely assume that  $Z=Y\oplus_F X$, where $F:X\to Y$ is quasilinear with $d(2;F)\leq 1$. By the induction hypothesis we have $c(m;X)\lesssim\log^{k-2}(m)$. In particular,
$$
c(2^n;X)\leq M n^{k-2}.
$$
for some $M$ and all $n$. Applying the second part of Lemma~\ref{lem:c and d} and taking into account that $c(3;Y)=1$, we obtain
%Aplicando el {\bf Lema B} tenemos
\begin{equation*}
    d(2^{n+1}; F) \leq c(2^n;X)+ d(2^{n}; F) \leq M n^{k-2} + d(2^{n}; F) ,
\end{equation*}
from which it follows that %de donde es evidente que
\begin{equation*}
    d(2^{n}; F) \leq  M[(n-1)^{k-2} + (n-2)^{k-2}+ \dots + 1^{k-2}] +1 \leq M' n^{k-1},
\end{equation*}
by Faulhaber's formula; see \cite[pp. 108--109, ``Power sums'']{aigner}. Hence $d(m;F)\leq M''\log^{k-1}(m)$ for some $M''$ and all $m$. Finally, to estimate
 $c(n;Y\oplus_F X)$ we may use the first part of Lemma~\ref{lem:c and d}, taking into account that $c(n;Y)=1$:
$$
c(n;Y\oplus_F X)\leq d(n;F)+ c(n;X)\leq  M''\log^{k-1}(n) +  M\log^{k-2}(n).
$$

To prove (b) we proceed by induction on $k$. The case $k=0$ is trivial. Assume it is known that the objects in $(\mathscr{S\hspace{-3pt}QE})^{k-1} \mathbf{B}$ have their constants dominated by $\log^{A}n$, where $A = 2^{k-1}-1 + \varepsilon$. Let $X, Y$ be two such spaces and assume that $Z$ fits in a short exact sequence $Y \to Z \to X$. As before, we may assume that $Z = Y \oplus_F X$, with $F: X \to Y$ quasilinear and $d(2; F) \le 1$. Since $\varepsilon>0$ is arbitrary, a quick glance at Lemma~\ref{lem:c and d}(a) shows that it suffices to see that
$$
d(n;F)\lesssim \log^{2^k-1+3\varepsilon}(n)
$$
for $n=2^j$, that is,
$$
d(2^j;F)\lesssim j^{2^k-1+3\varepsilon}.
$$
With an eye in Lemma~\ref{lem:c and d}(c) note that
$$
c(2^l;Y)\lesssim l^A\,;\quad c(j;Y)\lesssim \log^A j\,; \quad c(2^{j-l-1};X)\lesssim (j-l)^A.
$$
Now,
$$
d(2^j; F) \le c(j; Y) \sum_{0\leq l <j} c(2^l; Y) c(2^{j-l-1}; X)\lesssim \log^A j  \sum_{0\leq l <j} l^A (j-l)^A.
$$
But
$$\sum_{l=0}^{j-1} l^A (j-l)^A \le j^A \sum_{l=0}^{j-1} l^A \lesssim j^{2A+1}$$ again by Faulhaber's formula. Therefore:
% END CODEX CHANGE C40
$$ d(2^j; F) \lesssim (\log j)^A \; j^{2A+1} \quad\implies\quad c(n;Y\oplus_F X) \lesssim (\log \log n)^A (\log n)^{2A+1}.$$
Since $A = 2^{k-1}-1+\varepsilon$, we have $2A+1 = 2^k - 1 + 2\varepsilon$. On the other hand, $(\log \log n)^A\lesssim \log^\varepsilon n$ and thus
$$
 c(n;Y\oplus_F X) \lesssim (\log n)^{2A+1+\varepsilon}=
 (\log n)^{2^k-1+3\varepsilon}.
$$
This concludes the induction and yields the required upper estimates for $Y\oplus_F X$. \end{proof}

If one wants an epsilon-free bound in Part (b), observe that
if $c(n;X)\lesssim (\log n)^a$ and
$c(n;Y)\lesssim (\log n)^b$, Lemma~\ref{lem:c and d} evaluated at
$n=2^j$ yields
$$
c(n;Y\oplus_FX)\lesssim
(\log n)^{a+b+1}\bigl[\log(\log n)\bigr]^b.
$$

\section{Some special Orlicz spaces --- Proof of $(2) \Longrightarrow (3)$} % and the implication $(2) \Longrightarrow (3)$}\label{sec:meets}
%This section is devoted to proving the implication $(2) \Longrightarrow (3)$ of the main result.
We shall be using from now on the ``Orlicz'' functions
\begin{equation}\label{eq:phi_k}
\varphi_k(t)=t\log^k(1/t)
\end{equation}
to estimate sizes.
Although we are primarily concerned with the case where
$k\geq 1$
is an integer, most of this section makes sense for real $k\geq 0$.

However, even if we are mostly interested in the asymptotic behaviour for small values of $t$, these functions present some difficulties: they are increasing only up to $e^{-k}$, where they take the value $(k/e)^k$; they vanish at $t=1$ and may even be negative thereafter. Note that, at least, they all have $e^{-1}$ as a fixed point; see Figure ({\sc a}). So, it turns out that our Orlicz functions are not Orlicz after all.
Thus, proceeding as Pierre Menard
did with his rook’s pawn, we define once and for all, for each  real number $k>0$, the functions
\begin{equation*}
f_k(t) = t \log^k\left(e^k+\frac{1}{t}\right)\, ;
\end{equation*}
see Figure ({\sc b}).
These are continuous, increasing, concave and subadditive functions vanishing only at
$0$ --- for $k\geq 1$ they are even submultiplicative. In particular one has the $\Delta_2$-condition satisfied for $M=2$, namely: $f_k(2t)\leq 2f_k(t)$ for all $t\geq 0$. Hence each $f_k$ is an Orlicz function in the usual non-locally convex sense; see \cite[Section 2]{kalton77}, \cite[Sections 0.2--0.3]{turpin} and \cite[Section 1.3]{rolewicz} for further explanations. Clearly, for each fixed $k\geq 0$ one has $\varphi_k(t)\sim f_k(t)$ for $0\leq t\leq e^{-1}$,
which allows us to keep using the original versions
in most of the calculations.
%\vspace{-10pt}
\begin{figure}[htbp]
    \centering
    % Primera subfigura
    \begin{subfigure}[b]{0.475\textwidth}
        \centering
        \includegraphics[width=\textwidth]{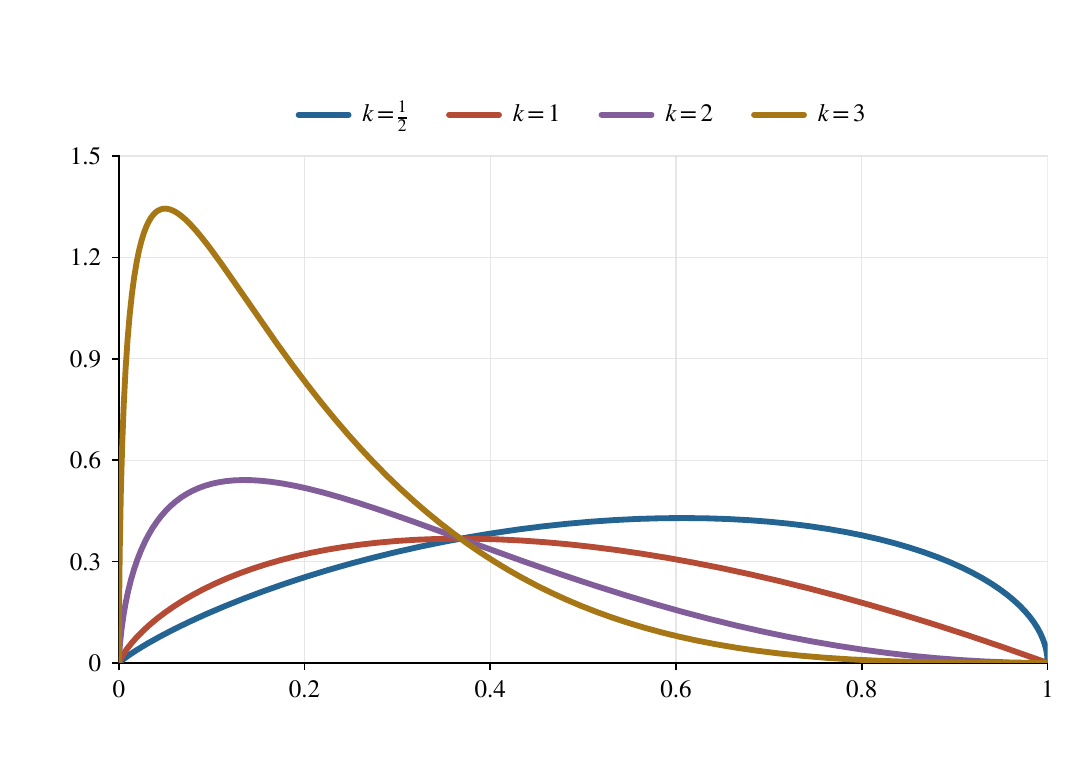} % Sin extensión o con .png/.pdf/.jpg
        \caption{The original functions $\varphi_k$}
        \label{fig:izq}
    \end{subfigure}
    \hfill % Este comando añade un espacio flexible entre las dos imágenes
    % Segunda subfigura
    \begin{subfigure}[b]{0.475\textwidth}
        \centering
        \includegraphics[width=\textwidth]{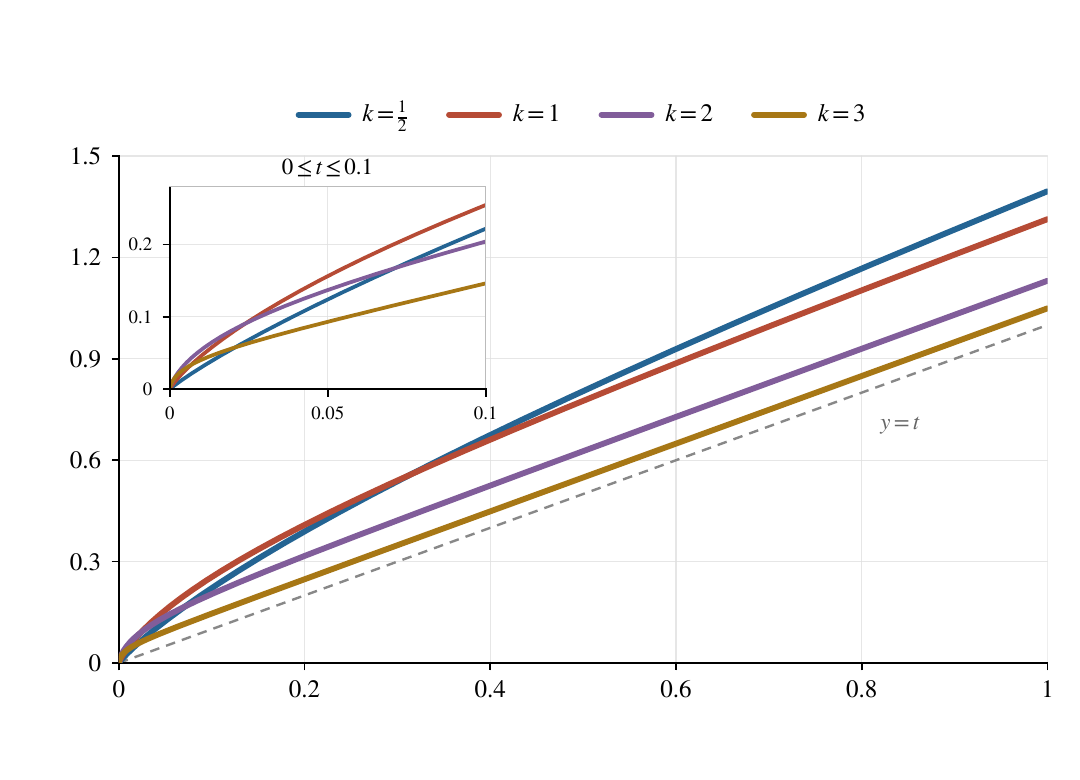}
        \caption{The graphs of $f_k$, scaled by $k^{-k}$.}
        \label{fig:der}
    \end{subfigure}
    %\caption{Título general que abarca las dos figuras conjuntas.}
    \label{fig:conjunta}
\end{figure}

The Orlicz sequence space  $\ell(f_k)$ associated to $f_k$ is the vector space of all real or complex sequences $x=(x_n)$ such that
$$
|x|_{f_k} =  \sum_{n\geq 1} f_k(|x_n|)<\infty.
$$
The function $(x,y)\mapsto|y-x|_{f_k}$ defines an invariant metric
on $\ell(f_k)$ which makes it a complete metric linear space. On the other hand, it is easy to find a constant $c>0$ (taking $c=e^{-k}$ is an overkill) such that
$$
\sup_{t>0}\frac{f_k(ct)}{f_k(t)}<1.
$$
Criteria from the year dot (\cite[Proposition 4.2]{kalton77} or \cite[Theorem 3.3.3]{rolewicz}) imply that $\ell (f_k)$ is a quasi-Banach space and that its  topology can be recovered from the  Luxemburg quasinorm
\begin{equation}\label{eq:luxlfk}
\|x\|_{f_k} =  \inf\left\{t>0 : \sum_{n\geq 1}f_k\left(\frac{|x_n|}{t}\right)\leq 1\right\}.
\end{equation}
With a slight abuse of notation, we may also consider the quantity (see the proof of Proposition~\ref{prop:lfk is serre})
\begin{equation*}
|x|_{\varphi_k} =  \sum_{n\geq 1} \varphi_k(|x_n|) = \|\varphi_k\circ |x|\|_1.
\end{equation*}
Note that, for each fixed $k$, we have $|x|_{\varphi_k}\sim |x|_{f_k}$ provided $|x_n|\leq e^{-1}$ for all $n$.

\medskip

We pave our way to furnish these spaces with more explicit quasinorms with the following Lemma. Fix $k>0$.
For each decreasing sequence  $x=(x_n)$ in $[0,1]$ define
\begin{align*}
 S_k(x)&=
    \sum_n x_n(1+\log^k(1/x_n)),\\
     L_k(x)&=
    \sum_n x_n(1+\log^k n).
\end{align*}
Observe that if either $L_k(x)\leq 1$ or $S_k(x)\leq 1$ then $\|x\|_1\leq 1$.
\begin{lemma}$\;$
\begin{itemize}
    \item [{\rm (a)}] If $\,\|x\|_1\leq 1$ then $L_k(x)\leq S_k(x)$.
    \item [{\rm (b)}] There is a constant $M=M(k)$ such that $S_k(x)\leq M$ whenever $L_k(x)\leq 1$.
\end{itemize}
\end{lemma}

\begin{proof} (a) We have $\sum_n x_n\leq 1$, so $x_n\leq 1/n$ hence $n\leq 1/x_n$ and $L_k(x)\leq S_k(x)$. To prove (b) assume $L_k(x)\leq 1$. Since $\sum_nx_n\leq L_k(x)$ it suffices to find an upper bound for
$$
\sum_n x_n\log^k(1/x_n)
$$
Split $\mathbb N =I\cup J$ with $I=\{n: x_n\geq n^{-2}\}$ and $J=\{n: x_n< n^{-2}\}$. If $n\in I$ then $n^2\geq 1/x_n$ and $\log (n^2)\geq \log(1/x_n) $, hence
$$
\sum_{n\in I} x_n\log^k(1/x_n)\leq \sum_{I} x_n\log^k(n^2)\leq 2^k L_k(x).
$$
To finish, consider the ``original'' function $\varphi_k(t)=t\log^k(1/t)$ for $t>0$. We already know
from the beginning of this section that:
\begin{itemize}
    \item $\varphi(t)\to 0$ as $t\to 0^+$,
    \item $\varphi$ is increasing for $0<t<e^{-k}$, and $\varphi(e^{-k})=k^ke^{-k}$.
\end{itemize}
If $n\in J$, then $\varphi(x_n)<\varphi(n^{-2})$, provided $n^{-2}$ is in the interval where $\varphi$ increases, that is, when $n\geq e^{k/2}$. Therefore,
$$
\sum_{n\in J} x_n\log^k\frac{1}{x_n}= \sum_{\substack{n\leq e^{k/2}\\ n \in J}} x_n\log^k\frac{1}{x_n}+ \sum_{\substack{n > e^{k/2}\\ n \in J}} x_n\log^k\frac{1}{x_n}
\leq e^{k/2} \frac{k^k}{e^k}+ \sum_{n\geq 1}\frac{1}{n^2}\log^k (n^2),
$$
%\begin{eqnarray*}
%\sum_{n\in J} x_n\log^k(1/x_n)&=& \sum_{\substack{n\leq e^{k/2}\\ n \in J}} x_n\log^k(1/x_n)+ \sum_{\substack{n > e^{k/2}\\ n \in J}} x_n\log^k(1/x_n)\\
%&\leq& e^{k/2} \frac{k^k}{e^k}+ \sum_{n\geq 1}\frac{1}{n^2}\log^k (n^2),
%\end{eqnarray*}
and the series converges for every $k$. Hence, it is enough setting
$$
M= 1+2^k+\left(\frac{k}{\sqrt{e}}\right)^k+ 2^k  \sum_{n\geq 1}\frac{\log^k n}{n^2}.\qedhere
$$
\end{proof}

From now on $x^*$ denotes the decreasing rearrangement of $x$ (of $|x|$, actually), which is well-defined if $x_n\to 0$.

\begin{corollary}\label{cor:lorentz}
Fix $k>0$. The following quasinorms are equivalent to the Luxemburg quasinorm on $\ell(f_k)$:
\begin{itemize}
    \item $\|x\|=\sum_n|x_n|(1+\log^k(\|x\|_1/|x_n|))$.
    \item $\|x\|_*=\sum_nx_n^*(1+\log^k n)$.
\end{itemize}
\end{corollary}

\begin{proof}
Both expressions are homogeneous and dominate the $\ell_1$-norm. The preceding lemma shows that they are equivalent ---  although we do not yet know if they are quasinorms.

It remains to compare these expressions with the Luxemburg quasinorm.
First note that for $t>0$ one has $\log^k(e^k+1/t)\geq \tfrac{1}{2}\big( k^k+\log^k(1/t)\big)$ and that $\min \{k^k: k>0\} =1/e$.

Now, if $x\geq 0$ has $\|x\|_{f_k}\leq 1$, then
$\sum_i f_k(|x_i|)\leq1$, hence
\begin{align*}
1\geq \sum_n x_n \log^k\left(e^k+\frac{1}{x_n}\right)\ge
\frac{1}{2e} \sum_n x_n +  \frac{1}{2}\underbrace{\sum_n x_n  \log^k\left(\frac{1}{x_n}\right)}_{(\star)}
\end{align*}
In particular $\|x\|_1\leq 2e$, so that
$$
(\star)\geq \sum_n x_n  \log^k\left(\frac{\|x\|_1}{2ex_n}\right),
$$
which shows that $\|\cdot\|\lesssim \|\cdot\|_{f_k}$. We now prove that $\|\cdot\|_{f_k}\lesssim \|\cdot\|_*$
Pick any decreasing $x\geq 0$ with $\|x\|_1\leq 1$. Then $x_n\leq 1/n$, that is, $1/x_n\ge n$.
Taking into account the elementary inequality
$(a+b)^k\leq2^{k-1}(a^k+b^k)$,
one has
\begin{equation*}
\begin{gathered}
\sum_n f_k(|x_n|)=
\sum_n x_n \log^k\left(e^k+\frac{1}{x_n}\right)\leq \sum_n x_n \log^k\left(e^k+n\right)\\
\leq \sum_n x_n \log^k\left(2e^k n\right)
\leq 2^{k-1}\sum_n x_n\left((2+k)^k + \log^k n\right)
\leq (2(2+k))^k\|x\|_*\, .\qedhere
\end{gathered}
\end{equation*}
\end{proof}

The corollary means that ``our'' Orlicz space $\ell(f_k)$
agrees with the ``Lorentz sequence space'' $d\big(1+\log^k; 1\big)$, which is simply delicious; see \cite[Section 4.e]{Lind-Tzaf} for the definition of a Lorentz space and \cite[Theorem 3.6]{kalton81} for the case $k=1$, where the Riemann $\zeta$-function plays a role! Note that, in these clothes, the unit vectors $e_i$ (the sequence with 1 in the $i$-th entry and 0 elsewhere) have size $1$, \emph{comme il faut}. From now on we shall use the quasinorm that appears in the first part of the corollary on $\ell({f_k})$ without further mention. Where strictly necessary, we shall include $\ell({f_k})$
in the subscript. This notation does not conflict with \eqref{eq:luxlfk}.

\begin{lemma} For each fixed $k$ one has  $\log^k n \lesssim c(n;\ell({f_k}))$.
\end{lemma}

\begin{proof} Let $(e_i)$ be the unit basis of $\ell(f_k)$ and put $s_n=e_1+\dots + e_n$. We have $\|e_i\|=1$ for all $i$, while
$$
\|s_n\|=\sum_{i=1}^n (1+\log^k(n)) = n (1+\log^k(n)).\qedhere
$$
\end{proof}
We shall see later that, moreover, $c(n;\ell({f_k})) \lesssim \log^k n$.\medskip

The following result grants us the greatest possible flexibility when choosing an equivalent $p$-norm --- and suggests that we are heading in the right direction, too. We present a proof that introduces certain ideas to be used later. Given $0<p<1$, a quasi-Banach space $X$ is said to be $p$-convex if there is a constant $C$ such that $\left\|\sum_i x_i\right\|\leq C  \left(\sum_i \|x_i\|^p\right)^{1/p}$ for every finite subset of $X$. It is clear that a space is $p$-convex if and only if its quasinorm is equivalent to a $p$-norm. %, and more often than not "$p$-convex" and "$p$-Banach" are used as synonyms (cf. \cite[p.7]{kpr}). We will distinguish then when the constant $C$ above has to be considered.

\begin{lemma}\label{lem:log->p} If $X$ is a quasi-Banach space with $c(n;X)\lesssim \log^k n$ for some $k\geq 0$ then $X$ is $p$-convex for every $0<p<1$.
\end{lemma}

\begin{proof} Fix $p\in(0,1)$. By the Aoki--Rolewicz theorem we may assume that the quasinorm of $X$ is an $r$-norm for some $0<r<1$. Assume that $\sum_i \|x_i\|^p\leq 1$ and $\|x_1\|\geq \|x_2\|\geq \dots$ Then $\|x_i\|\leq i^{-1/p}$. Adding some zeros if necessary we may assume that we are dealing with $2^N-1$ points:
\begin{equation}\label{eq:2N-1}
x_1, x_{2^1}, x_3, x_{2^2}, x_5,\dots, x_{2^3}, \dots \overbrace{x_{2^n}, \dots, x_{2^{n+1}-1}}^{\text{$n$-th group}}, x_{2^{n+1}}, \dots x_{2^{N-1}}\dots x_{2^N-1}.
\end{equation}
Take $M$ such that $c(n;X)\leq M\log^k_2 n$ for all $n\geq 2$  and let us first bound the quasinorm of the sum of the terms in the $n$-th group ($n>0$):
$$
\left\|\sum_{2^{n}\leq i<2^{n+1}} x_i\right\|\leq c(2^{n}; X)\sum_{2^{n}\leq i<2^{n+1}} \|x_i\|
\leq Mn^k \sum_{2^{n}\leq i<2^{n+1}} \frac{1}{i^{1/p}}
\leq \frac{Mn^k 2^{n}}{(2^{n})^{1/p}}= \frac{Mn^k}{ (2^{1/p-1})^{n}},
$$
which is great because $p<1$ implies $2^{1/p-1}>1$ and then  the series $\sum_{n=1}^\infty n^{kr}(2^{1-1/p})^{rn}$ is summable for every $r>0$.
Consequently,
\begin{equation*}
\Big\|\sum_{1\leq i<2^N}\hspace{-0,3cm} x_i\Big\|^r =
\Big\|\sum_{0\leq n<N}\;\sum_{2^{n}\leq i<2^{n+1}} \hspace{-0,3cm} x_i\Big\|^r
\leq \sum_{0\leq  n<N}\Big\|\sum_{2^{n}\leq i<2^{n+1}} \hspace{-0,3cm} x_i\Big\|^r
 \leq M^r  \left(1+ \sum_{n\geq 1}  \frac{n^{kr}}{(2^{1/p-1})^{rn}}\right).\qedhere\end{equation*}
\end{proof}

Recall that one of our ultimate aims is to show that every space  whose sum constants are bounded by a power of the logarithm is a quotient of some $\ell(f_l)$. Our point of entry is the following summability argument, which is an imperfect generalization of Kalton's result
\cite[Theorem 7.2]{kalton81} quoted at the Introduction.

\begin{lemma}\label{lem:galb} Let $X$ be a quasi-Banach space such that
$c(n;X)\lesssim \log^k n$. If $l>k$, there is a constant
$C=C(X,k,l)$ such that, for every finite family $x_1,\dots,x_N$ in $X$ with $\|x_1\|\geq \|x_2\|\geq \dots$ one has
\begin{equation}\label{eq:||<M}
\left\|\sum_{i=1}^N x_i\right\|
\leq C\sum_{i=1}^N \|x_i\|\bigl(1+ \log^l i \bigr). %,
%\qquad a_i=\|x_i\|.
\end{equation}
%\begin{equation}\label{eq:||<M}
%\left\|\sum_{i=1}^N x_i\right\|
%\leq C\sum_{i=1}^N a_i^*\bigl[1+(\log i)^l\bigr],
%\qquad a_i=\|x_i\|.
%\end{equation}
Consequently, if $(\|x_i\|)$ belongs to $\ell(f_l)$, then
$\sum_i x_i$ converges in $X$ and obeys the corresponding estimate.
\end{lemma}

\begin{proof}
Choose $r$ with
$$
\frac1{1+l-k}<r<1
$$
and use Lemma~\ref{lem:log->p} to replace the quasinorm of $X$ by an
equivalent $r$-norm. %Permuting a finite sum if necessary, assume
%$a_1\geq a_2\geq\cdots$.
Now we proceed as in the proof of Lemma~\ref{lem:log->p}, grouping the summands as in \eqref{eq:2N-1}:
for $j\geq 0$ set
\begin{itemize}
\item $I_j=\{i:2^j\leq i<2^{j+1}\}$;
\item $y_j=\sum_{i\in I_j}x_i$;
\item $u_j=\sum_{i\in I_j}\|x_i\|\bigl(1+ \log^l i \bigr).$
\end{itemize}
On $I_j$ one has $1+(\log i)^l < 1+ \log^l 2 (1+j)^l$. Hence
$$
\|y_j\|\leq c(2^j;X)\sum_{i\in I_j}\|x_i\|
\lesssim(1+j)^{k-l}u_j.
$$
The $r$-triangle inequality and H\"older's inequality, with exponents
$1/r$ and $1/(1-r)$, give
%\begin{align*}
%\left\|\sum_i x_i\right\|^r
%&\leq\sum_j\|y_j\|^r\\
%&\lesssim\sum_j(1+j)^{-(l-k)r}u_j^r\\
%&\leq\left(\sum_j u_j\right)^r
%\left(\sum_j(1+j)^{-(l-k)r/(1-r)}\right)^{1-r}.
%\end{align*}
\[
\left\|\sum_i x_i\right\|^r
\leq\sum_j\|y_j\|^r\lesssim\sum_j(1+j)^{-(l-k)r}u_j^r\leq\left(\sum_j u_j\right)^r
\left(\sum_j(1+j)^{-(l-k)r/(1-r)}\right)^{1-r}.
\]
The last series converges by the choice of $r$. Taking $r$-th roots proves
\eqref{eq:||<M}. Applying the estimate to differences of partial sums, and
using the density of the finitely supported sequences in $\ell(f_l)$, shows
that the partial sums of any family with $(\|x_i\|)\in\ell(f_l)$ are Cauchy. %is summable in $X$$.
 The completeness of
$X$ finishes the proof.
\end{proof}

To establish the results in full generality, we may use the following nonseparable versions of the preceding spaces. Let $I$ be a set and put
$$
\ell(I,f_k)=\left\{a\in\mathbb K^I:
\sum_{i\in I}|a_i|\left(1+\log^k\frac{\|a\|_1}{|a_i|}\right)<\infty
\right\}
$$
with the obvious quasinorm. Every element of $\ell(I,f_k)$, or $\ell_1(I)$, has countable support. The following result shows that the spaces $\ell(I,f_l)$ are a kind of free objects. Its proof is clear from Lemma~\ref{lem:galb}:

\begin{proposition}\label{prop:orlicz-operators}
    Let $X$ be a quasi-Banach space such that  $c(n;X)\lesssim \log^k n$ and let $I$ be any set. If $l>k$, then every bounded function $x:I\to X$ extends to an operator $u: \ell(I,f_l)\to X$ in the form $u(a)=\sum_{i\in I} a_i x(i)$.\end{proposition}

The inevitable consequence is the implication $(2)\Longrightarrow(3)$ we
are looking for.

\begin{corollary}\label{2implies3}
If $c(n;X)\lesssim \log^k n$ and $l>k$, then $X$ is a quotient of
$\ell(I,f_l)$ for some index set $I$. \end{corollary}

\begin{proof} Let $I$ be a dense subset of the unit ball of $X$. The operator
$T: \ell(I,f_l)\to X$ given by $Te_u=u$ is a quotient map. \end{proof}

Turpin, whose work on the theory of non-locally convex spaces remains among the finest ever, introduced the concept of a \emph{galbe} in \cite[Chapitre II]{turpin}. In Turpin's terminology, Lemma~\ref{lem:galb} means  that $X$ is galbed by $\ell(f_l)$. Note that if $Z$ is galbed by a quasi-Banach sequence space $E$, then so is any subspace or quotient of $Z$, and $E$ has the lifting property that every operator from $E$ to a quotient  of $Z$ lifts to $Z$:
$$
\xymatrix{ & & Z \ar[d]^{\mathrm{quotient}} \\
E \ar[rru]^{\mathrm{lifting}} \ar[rr]_{\mathrm{operator}}& & Q}
$$

\section{Interpolation of quasi-Banach spaces --- Proof of $(3) \Longrightarrow (1)$}\label{3implies1}

Following Pierre Menard's lead once again, this section includes a concise, ultimately superfluous, account of the basic complex interpolation facts needed to ease the reader's understanding that for each $k\in\mathbb N$ the space $\ell(f_k)$ introduced in the previous section is a well-located subspace of an iterated extension of Banach spaces. It is superfluous because in the specific situation we will work with, one might ignore
almost everything we are about to explain and simply apply Hadamard's three-line theorem \cite{hadamard} coordinate by coordinate.

%all the plurisubharmonic stuff and just focus on a coordinate-by-coordinate interpolation.

\medskip

Until further notice, the ground field is $\mathbb C$, the complex numbers. While complex interpolation theory is well established for Banach spaces, complex interpolation for quasi-Banach spaces presents additional difficulties stemming from the failure of the maximum modulus principle. Fortunately, our situation is far more favourable
because we will be interpolating inside a Banach space, in fact, a Hilbert one.

% since the usual quasinorm of the quasi-Banach spaces $L_p, \ell_p$ carry plurisubharmonic (PSH) quasinorms, which for practical purposes means that the maximum modulus principle holds; cf. Kalton--Mitrea \cite[Section~3]{kalton-mitrea}.\medskip

If $Y$ is a quasi Banach space and $U$ is an open set of the complex plane, a function $f:U\to Y$ is said to be analytic if for every $z_0\in U$ there is $r>0$ and a sequence $(y_n)_{n\geq 0}$ in $Y$ such that
$f(z)=\sum_{n=0}^\infty y_n(z-z_0)^n $ uniformly on $|z-z_0|\leq r$. This is equivalent to complex differentiability when $Y$ is a Banach space, but not in general. If the reader is unfamiliar with the theory of analytic functions with values in quasi Banach spaces, we recommend consulting \cite[Section 3]{kalton-mitrea}, \cite{kalton-analytic}, \cite{kalton-pluri} or \cite[Chapitre IX]{turpin}.

\adef A compatible couple of quasi-Banach spaces consists of the following data:
\begin{itemize}
    \item[(a)] Two quasi-Banach spaces $X_0, X_1$ and a Hausdorff topological vector space $W$.
    \item[(b)] For $k=0,1$, an injective continuous operator $\jmath_k: X_k\to W$.
\end{itemize}
\zdef

It is standard to say that $(X_0, X_1)$ is a compatible couple, omitting from the notation the most essential part: the inclusions into the ambient space, which are understood. Let us leave it as is. More often than not, $W$ is replaced by the sum space $\Sigma=\{x\in W: x=x_0+x_1 \text{ with $x_k\in X_k; k=0,1$}\}$ endowed with the quasinorm
$$
\|x\|_\Sigma=\inf\{  \|x_0\|_{X_0}+\|x_1\|_{X_1}:   x=x_0+x_1, \text{ with $x_k\in X_k; k=0,1$}\}.
$$
and the obvious operators $X_k\to \Sigma$.\smallskip

%Let $(X_0, X_1)$ be a compatible couple of quasi-Banach spaces sitting in $W$.
 %We assume from now on that both $X_0$ and $X_1$ carry PSH quasinorms. The only couple we really need here is $X_0=\ell_p$ and $X_1=\ell_q$, with $0<p<q<\infty$, so that we may take $W=\ell_q$ and the formal inclusions. At one instance, we will also consider the case $X_0=L_p, X_1=L_q$ with $W=L_p$ and the formal inclusions.

We now introduce the spaces of analytic functions we require, keeping things as simple as possible, but no simpler. Let $\mathbb S= \{z\in\mathbb C: 0<\operatorname{Re}z<1\}$ be the unit strip and $\overline{\mathbb S}$ its closure. Also, for $k=0,1$, let $\partial_k =\{k+yi: y\in\mathbb R\}$ the corresponding half of the boundary of $\mathbb S$. Given a compatible couple $(X_0,X_1)$ with sum $\Sigma$ we denote by $\mathscr A = \mathscr A(X_0,X_1)$ the space of those bounded analytic functions $f:{\mathbb S}\to \Sigma$ satisfying the following additional conditions:
\begin{itemize}
    \item $f$ admits a continuous (hence bounded) extension $\overline{\mathbb S}\to \Sigma$, which we denote again by $f$.
 %   \item The restriction of $f$ to $\mathbb S$ if analytic.
    \item For $k=0,1$, the restriction of (the extension of) $f$ to $\partial_k$  is continuous and bounded from $\partial_k$ to $X_k$.
\end{itemize}

We equip $\mathscr A(X_0,X_1)$ with the quasinorm
$$
\|f\|_{\mathscr A}=\max\left\{ \sup_{k=0,1; z\in \partial_k} \|f(z)\|_{X_k}, \sup_{z\in\mathbb S}\|f(z)\|_\Sigma\right\},
$$
and quote Kalton--Mitrea \cite[p.~3913]{kalton-mitrea}: \emph{The second term in the definition of $\|\cdot\|_{\mathscr A}$  is precisely designed to ensure both
the completeness of $\mathscr A$ and the continuity of the evaluation functions. However, it
should be pointed out that this term can be dropped as being dominated by the
first if $\Sigma$ is A-convex.} (Here, a quasi-Banach $X$ is said to be A-convex, for analytically convex, if there is a constant $C$ such that for every ``polynomial'' $P:\mathbb C\to X$ one has $\|P(0)\|\leq C\max_{|z|=1}\|P(z)\|$.) This applies to the couple used below whose sum space is even a Banach space. The following result isolates
the key features of $\mathscr A$:

 \begin{lemma}$\;$
 \begin{itemize}
\item[(a)] For each $z\in \mathbb S$ and every $k\in \mathbb N$, the evaluation map $\delta_z:\mathscr A\to \Sigma$ and the evaluation of the $k$-th derivative map $\delta_z^{(k)}: \mathscr A\to \Sigma$ at $z$ are bounded.
\item[(b)] For every conformal equivalence
$\varphi:\mathbb S\to\mathbb D$ and every analytic function $f:\mathbb S\to \Sigma$ we have
$f\in\mathscr A$ if and only if $\varphi \cdot f\in\mathscr A$ and
$\|\varphi \cdot f\|_{\mathscr A}= \|f\|_{\mathscr A}$.
\end{itemize}
\end{lemma}

\begin{proof}[Hint for the proof]
(a) is as trivial as it is obvious for $k=0$. For $k\geq 1$ see \cite[Lemma 2.4]{ccc} or the last paragraph of Section~2 of \cite{cck}. Part (b) ultimately depends on the structure of the strip: every conformal map $\varphi:\mathbb S\to\mathbb D$ has a continuous extension $\varphi:\overline{\mathbb S}\to \overline{\mathbb D}$ which maps the boundary of $\mathbb S$ to $\mathbb T$ and omits exactly two points; see \cite[Section 8.1]{pisier}. From this, the reader should have no difficulty in completing the proof.
\end{proof}

The lemma yields that, for every $z \in \mathbb S$, the space $\mathscr A$ is isometric to its subspace $\operatorname{Ker}\delta_z = \{f \in \mathscr A : f(z) = 0\}$, the isometry being given by multiplication by any conformal map $\varphi:\mathbb S\to\mathbb D$ vanishing at $z$. Given $0<\theta<1$, we now define the interpolation space
$$X_\theta= (X_0, X_1)_\theta = \{x\in W: x=f(\theta) \text{ for some $f\in\mathscr A$}\}$$
and endow it with the quotient quasinorm, namely
$$
\|x\|_{X_\theta}= \inf_{x=f(\theta)} \|f\|_{\mathscr{A}} \,.
$$
Note that if $\delta_\theta:\mathscr A\to W$ is the evaluation map, then $X_\theta$ is isometric (but not equal) to the quotient of $\mathscr A$ by $\Ker \delta_\theta$. For example, for $0<p,q<\infty$, considering  $\ell_p$ and $\ell_q$ embedded in $\ell_\infty$ by the formal identity, we have
$$(\ell_p, \ell_q)_\theta = \ell_r, \quad \text{where}\quad
\frac{1}{r}= \frac{1-\theta}{p}+\frac{\theta}{q};
$$
see \cite[Theorem 3.4]{kalton-mitrea} for a much more general result. In particular, if $X_1=\ell_2$, we can set $X_0=\ell_{2/3}$ to get $X_{\theta}=\ell_1$ at $\theta=1/2$. This is the only couple that we need here.

\subsection{Derived spaces} Following Rochberg \cite{rochberg}, for each $n\in\mathbb N$, we define the $n$-th derived space
\begin{equation*}
    X_\theta^{(n)} = \left\{(x_{n-1}, \dots, x_0)\in \Sigma^n:
\text{ there is $f\in\A$ such that $x_k= \dfrac{f^{(k)}(\theta)}{k!}$ for $0\leq k<n$}     \right\}\,
\end{equation*}
endowed with the quasinorm  $\|(x_{n-1}, \dots, x_0)\|_{ X_\theta^{(n)}}=\inf \|f\|_{\A}$. As before, $X_\theta^{(n)}$ is isometric to the quotient of $\A$ by its closed subspace $\bigcap_{k<n}\ker \delta^{(k)}_\theta$ and thus it is a quasi-Banach space. Although the long-suffering reader might think that we are going from bad to worse since none of these things seems to have anything to do with Serre closures, that is not the case: the derived spaces are iterated extensions of the interpolation spaces $X_\theta=X_\theta^{(1)}$ in the following sense: given $n,k\in\mathbb N$, the obvious sequence
\begin{equation*}%\label{dia:n n+k k}
\xymatrix{
0\ar[r] &X_\theta^{(n)} \ar[r]^-\imath &X_\theta^{(n+k)}\ar[r]^-\pi& X_\theta^{(k)} \ar[r] &0\\
%0\ar[r]&X_\theta^{(n-j)} \ar[d]\ar[r]&  X_\theta^{(m-%j)}\ar[d]\ar[r]& X_\theta^{(m-n)} \ar[r]& 0\\
%&0&0
}
\end{equation*}
is exact. Here, as one may guess, $\imath(x_n,\dots, x_1)=(x_n,\dots, x_1, 0,\dots, 0)$, with $k$ zeros appended, and $\pi(x_{n+k},\dots,x_{k+1}, x_k,\dots, x_1)= (x_k,\dots, x_1)$. It is the exactness that is not obvious. In fact, even the boundedness of the $\imath$
is far from being trivial; see \cite[Section~3]{cck}, especially Proposition~2 and Theorem~4. Once we know that, it is clear that when $X_\theta^{(1)}$ is a Banach space all the derived spaces $X_\theta^{(n)}$ belong to the Serre closure of the class of Banach spaces since they are iterated extensions of Banach spaces, with $X_\theta^{(n)}\in{\bf B}^{(n)}$. Those readers who have properly assimilated Section~\ref{subsec:quasilinear} will not be surprised that the exact sequences above can be described by quasilinear maps. We only need the case $k=1$. Fix $\varepsilon>0$ and, for each $x\in X_\theta$, choose $F_x\in \mathscr A$ such that $F_x(\theta)=x$, with $\|F_x\|_\mathscr A\leq (1+\varepsilon)\|x\|_{X_\theta}$, in such a way that $F_x$ depends homogeneously on $x$ in the sense that if $y=\lambda x$ for some $\lambda\in\mathbb C$, then $F_y=\lambda F_x$. This just means that the map $F_\bullet :X\to \mathscr A$ sending $x$ to $F_x$ is a bounded homogeneous section of the evaluation map $\delta_\theta :\mathscr A\to X_\theta$. The reason for this notation is that we prefer to reserve $F_x(z)$ for the value of $F_x:S\to \Sigma$ at the point $z\in \mathbb S$. The analytic function $F_x$ is often called an extremal for $x$.

\begin{lemma}\label{lem:rochberg-twisted} Fix $0<\theta<1$ and $n\in\mathbb N$. If $F_\bullet: X_\theta\to\mathscr A$  is a bounded homogeneous section of the evaluation map, then the map $\Omega_{1,n}:X_\theta \to\Sigma^n$ given by
$$
\Omega_{1,n}(x)=\left(\frac{F_x^{(n)}(\theta)}{n!},\ldots,
F_x'(\theta)\right)
$$
is quasilinear from $X_\theta$ to $X_\theta^{(n)}$ and   the twisted sum
$X_\theta^{(n)}\oplus_{\Omega_{1,n}}X_\theta$ agrees with $X_\theta^{(n+1)}$, up to equivalent quasinorms.
%$\|(x_n,\ldots,x_1,x_0)\|_{X_\theta^{(n+1)}}$ and $\|(x_n,\ldots,x_1)-\Omega_{1,n}(x_0)\|_{X_\theta^{(n)}} +\|x_0\|_{X_\theta}.$
\end{lemma}

In particular, for each fixed $\theta$ and $n$ one has (see \cite[Section 3.5, Proposition~7]{cck}):
\begin{equation}
\|(x_n,\ldots,x_1,x_0)\|_{X_\theta^{(n+1)}}\sim \|(x_n,\ldots,x_1)-\Omega_{1,n}(x_0)\|_{X_\theta^{(n)}} +\|x_0\|_{X_\theta}.
\end{equation}

%\medskip

When we particularize to the couple $(\ell_{2/3}, \ell_2)$ the spaces $\ell(f_n)$ magically appear as well placed subspaces of the corresponding $X_\theta^{(n+1)}$.

Since $\ell_{2}$ contains $\ell_{2/3}$ the sum space is just $\ell_2$. As we mentioned before,
for $0<\theta<1$ one has $X_\theta=\ell_p$, where
$\frac{1}{p}=\frac{3}{2}(1-\theta)+ \frac{1}{2}\theta$, which in particular implies that $X_{1/2}=\ell_1$.

Set $\theta=1/2$ and let  $x\in \ell_1$ be normalized and positive. Put
$$
F_x(z)=
x^{\frac{3}{2}-z}.
$$
It is clear that:
\begin{itemize}
    \item $F_x$ continuous on $\overline{\mathbb S}$ and  analytic on ${\mathbb S}$;
    \item $F_x(\tfrac{1}{2})=x$, and $\|F_x\|_{\mathscr A}=1$;
    \item actually one has
    \begin{equation*}
        F_x(z)=\sum_{n=0}^\infty \left(\frac{x}{n!}\log^n \frac{1}{x} \right)\big{(}z-\tfrac{1}{2}\big{)}^n
    \end{equation*}
    uniformly on $|z-\tfrac{1}{2}|\leq r $ for every $r<\tfrac{1}{2}$.
\end{itemize}
For arbitrary $x\in\ell_1$, let $x=u|x|$ be the polar decomposition and take
$F_x= u\|x\|_1 F_{|x|/\|x\|_1}$, so that
 \begin{equation*}
        F_x(z)=\sum_{n=0}^\infty \left(\frac{x}{n!}\log^n \frac{\|x\|_1}{|x|} \right)\big{(}z-\tfrac{1}{2}\big{)}^n,
    \end{equation*}
which clearly shows that the Taylor coefficients of the extremal $F_x$ are given by
$$\Phi_n(x)=\frac{x}{n!}\log^n \frac{\|x\|_1}{|x|} 	$$
for $n\geq 1$  and thus the differential map $\Omega_{1,n}$ is
$$
\Omega_{1,n}(x)= (\Phi_n (x),\dots, \Phi_1(x)).
$$
The connection with the original functions $\varphi_k$ in \eqref{eq:phi_k}
is immediately apparent: if $x\geq0$ is normalized in $\ell_1$, then $\Phi_k(x)=\varphi_k\circ x$.
Therefore, given $x\in\ell_1$ one has $\Phi_k(x)\in\ell_1$ if and only if $x\in\ell(f_k)$, with $\|\Phi_k(x)\|_1+\|x\|_1\sim \|x\|_{\ell(f_k)}$.

\begin{proposition}\label{prop:lfk is serre} For each integer $k$, the complex space  $\ell(f_{k})$ is isomorphic to the closed subspace
$\{(0,\dots, 0, x)\in {X_{1/2}^{(k+1)}}\}$. Consequently, $\ell(f_{k})$ belongs to the Serre closure of Banach spaces.
\end{proposition}

\begin{proof} Induction on $k$. For $k=0$ this is the basic interpolation identity $X_{1/2}=\ell_1$, with equivalent (actually equal) quasinorms. Assume the statement true for $0\leq k\leq n-1$ and let us prove it for $k=n$. As one can imagine we will check two inclusions and the corresponding bounds.\medskip

First, assume $x\in \ell(f_n)$. We must prove that $(0,\dots, 0, x)$ belongs to ${X_{1/2}^{(n+1)}}$ and that $\|(0,\dots, 0, x)\|_{X_{1/2}^{(n+1)}}\lesssim \|x\|_{\ell(f_{n})}$. Note that $\ell(f_k)\subset \ell(f_{k-1})\subset \ldots \subset \ell(f_1)\subset \ell_1$ and the inclusions are bounded.
One has
\begin{equation*}
\|\overbrace{(0,0, \dots, 0, x)}^{n+1\text{ entries}}\|_{X_{1/2}^{(n+1)}}\sim \|\Omega_{1,n}(x)\|_{X_{1/2}^{(n)}} + \|x\|_1 =  \overbrace{\|(\Phi_n (x),\dots, \Phi_1(x))\|_{X_{1/2}^{(n)}}}^{(\star)} + \|x\|_1.
\end{equation*}
Now (the implicit constant is $c(n;{X_{1/2}^{(n)}})$, if one wants to know):
\begin{gather*}
(\star)
\lesssim\|(\Phi_n (x),0, \dots,0)\|_{X_{1/2}^{(n)}}+
   \|(0, \Phi_{n-1} (x),0, \dots,0)\|_{X_{1/2}^{(n)}}+\dots+\|(0, \dots,0, \Phi_1 (x))\|_{X_{1/2}^{(n)}}
\\
\sim \|\Phi_n (x)\|_{X_{1/2}^{(1)}}+
   \|(0, \Phi_{n-1} (x))\|_{X_{1/2}^{(2)}}+\dots + \|(0, \dots,0, \Phi_2 (x))\|_{X_{1/2}^{(n-1)}} +\|(0, \dots,0, \Phi_1 (x))\|_{X_{1/2}^{(n)}}
 \end{gather*}
%By the way, the implicit constant is $c(n;{X_{1/2}^{(n)}})$ if one wants to know.
The first summand is  $\|\Phi_n (x)\|_1$, which is dominated by $\|x\|_{\ell(f_n)}$. The remaining terms are covered by the induction hypothesis:
$$
  \|(0, \dots,0, \Phi_i (x))\|_{X_{1/2}^{(n-i+1)}}\lesssim \|\Phi_i (x)\|_{\ell(f_{n-i})}.
$$
The terms on the right-hand side of the last display can be controlled through  $\|x\|_{\ell(f_n)}$ too. By
unconditionality and homogeneity, we may assume $x\geq0$ and $\|x\|_1=1$.
For $0<t\leq e^{-1}$ one has
$$
\left(t\log^i\frac{1}{t}\right) \log^{n-i}\frac{1} {t\log^i\frac{1}{t}} \leq t\log^n\frac{1}{t}
$$
(that is, $\varphi_{n-i}(\varphi_i(t))\leq \varphi_n(t)$,  except for the restriction on $t$). Hence, if we
assume that no entry of $x$ exceeds $e^{-1}$ (which is true except for at most two entries!), we have
$$
\left| \frac{x}{i!}\log^i\frac{1}{x}\right|_{\varphi_{n-i}}
\leq \left| {x}\log^i\frac{1}{x}\right|_{\varphi_{n-i}}
\leq \left\| \varphi_{n-i}\left({x}\log^i\frac{1}{x}\right)\right\|_1
= \|\varphi_n(x)\|_1 =|x|_{\varphi_n}
$$
We have shown that if $x$ is normalized in $\ell_1$ and has small quasinorm in $\ell(f_n)$ then $(0,\dots, 0, x)\in {X_{1/2}^{(n+1)}}$ with small (bounded) quasinorm. By homogeneity
$$\|(0,\dots, 0, x)\|_{X_{1/2}^{(n+1)}}\lesssim \|x\|_{\ell(f_{n})}$$
and we have the squonk in the sack, which is half of the battle.\medskip

To get something more than tears we need to prove the other
containment, namely that if $(0,\dots,0,x)$ belongs to $X^{(n+1)}_{1/2}$, then $x\in\ell(f_{n})$ since the open mapping theorem (and the first part) automatically provides the equivalence of the quasinorms. The fact that the points of the above form form a closed subset of the quasi-Banach space $X^{(n+1)}_{1/2}$ is obvious once one realizes that the natural map $X^{(m)}_{1/2}\to \ell_2\oplus \dots \oplus \ell_2$ is bounded. Of course we may assume (and do) that the statement is true for $0\leq k\leq n-1$ and prove it for $k=n$.

This part of the proof is the most demanding of the paper (and indeed it establishes its deepest result), but it ultimately relies on a rather simple fact, namely that for every $i,j,a,b,c,d>0$ one has
\begin{equation}\label{eq:limfifj}
\frac{\left(t\log^i\dfrac{a}{t}\right) \log^j\dfrac{b}{t\log^i(c/t)}}{t\log^{i+j}\dfrac{d}{t}} \longrightarrow 0%\qquad(\text{as }t\to 0^+),
\end{equation}
as $t\to 0^+ $, which essentially means that the composition $\varphi_j\circ \varphi_i$ is much smaller than $\varphi_{i+j}$ for small arguments.

Thus, assume that $(0,\dots,0,x)$ belongs to $X^{(n+1)}_{1/2}$.
Without loss
of generality we may assume that $x$ is positive and normalized in $\ell_1$. Moreover, for each $k<n$, the point $(0,\dots,0,x)$ (with $k$ zeroes preceding $x$) belongs to $X^{(k+1)}_{1/2}$ and the induction hypothesis yields that $x\in\ell(f_{k-1})$, with
$$
\|\Phi_k(x)\|_1\lesssim \|x\|_{\ell(f_k)}\sim \|(0,\dots, 0, x)\|_{X_{1/2}^{(k+1)}}.
$$
To prove that $x\in\ell(f_n)$ it suffices to see that $\Phi_n(x)\in \ell_1$. One has
%\begin{align*}
  $$\|(0,\dots, 0, x)\|_{X_{1/2}^{(n+1)}} \sim \overbrace{\|(\Phi_n x, \Phi_{n-1}x,\dots, \Phi_2x,\Phi_1x)\|_{X^{(n)}_{1/2}}}^{(\star)}+ \|x\|_1
  .
  $$
  %\\
  %&\sim
   %\|(\Phi_n x - \Phi_{n-1}\Phi_1x, \Phi_{n-1}x - \Phi_{n-2}\Phi_1x,\dots, \Phi_3x-\Phi_2\Phi_1x, \Phi_2x-\Phi_1\Phi_1x)\|_{X^{n-1}_{1/2}}+ \|\Phi_1x\|_1+\|x\|_1
%\end{align*}
Unrolling the first summand we obtain
$$
(\star) \sim  \|(\Phi_n x - \Phi_{n-1}\Phi_1x, \Phi_{n-1}x - \Phi_{n-2}\Phi_1x,\dots, \Phi_3x-\Phi_2\Phi_1x, \Phi_2x-\Phi_1\Phi_1x)\|_{X^{(n-1)}_{1/2}}+ \|\Phi_1x\|_1.
$$
One may continue beating the same drum, but after a few steps the formulas   become really messy --- and the term ``unrolling'' a bit ironic.

Thus, rather than looking for an intractable closed-form expression, we will give a recurrence relation simple enough to capture the essence of the process.

The procedure is as follows:

\begin{itemize}
    \item Start with $(y_n^{(1)},\dots,y_1^{(1)})= (\Phi_n x,\dots,\Phi_1x)$.

    % \item Start with $(y_n^{(1)},y_{n-1}^{(1)},\dots, y_2^{(1)},y_1^{(1)})= (\Phi_n x, \Phi_{n-1}x,\dots, \Phi_2x,\Phi_1x)$.

    \item Having defined $(y_n^{(k)},\dots, y_{k+1}^{(k)},y_k^{(k)})$, set
    $$(y_n^{(k+1)},\dots, y_{k+1}^{(k+1)})=
    (y_n^{(k)}-\Phi_{n-k} y_k^{(k)},\dots, y_{k+1}^{(k)}-\Phi_1y_k^{(k)}) .
    $$
    \item After $n-1$ steps one arrives at a certain vector $y_n^{(n)}$ in $\ell_1$.
\end{itemize}

Of course one has %(with a slightly abuse of notations)
\begin{align*}
     \|(0,\dots, 0, x)\|_{X_{1/2}^{(n+1)}} &\sim \|(y_n^{(1)},\dots,y_1^{(1)})\|_{X_{1/2}^{(n)}}+ \|x\|_1\\
     &\sim \|(y_n^{(2)},\dots,y_2^{(2)})\|_{X_{1/2}^{(n-1)}}+\|y_1^{(1)}\|_1  + \|x\|_1\\
     & \sim\dots\\
     &\sim \|(y_n^{(k)},\dots,y_k^{(k)})\|_{X_{1/2}^{(n-k+1)}}+\|y_{k-1}^{(k)}\|_1 +\dots +\|y_1^{(1)}\|_1  + \|x\|_1\\
     & \sim\dots\\
     &\sim \|(y_n^{(n-1)},y_{n-1}^{(n-1)})\|_{X_{1/2}^{(2)}}+\|y_{n-1}^{(n-2)}\|_1 +\dots +\|y_1^{(1)}\|_1  + \|x\|_1\\
     &\sim \|(y_n^{(n)}\|_1 +\|y_{n-1}^{(n-1)})\|_1+\|y_{n-1}^{(n-2)}\|_1 +\dots +\|y_1^{(1)}\|_1  + \|x\|_1.
\end{align*}
 The implicit constants depend only on $n$, but we don't need this fact; we only need to know that $\|y_n^{(n)}\|_1$ is finite, from which we will deduce that $\|y_n^{(1)}\|_1$ is finite thus concluding that $x$ belongs to $\ell(f_n)$. To see that we will prove that the output $y_n^{(n)}$ essentially dominates the starting vector $y_n^{(1)}=\Phi_n(x)$, and to prove it, we need to know how all the vectors evolve throughout the inductive process described above. The following piece ends the proof:

\medskip

\noindent{\bf Claim}\, For every $1\leq k\leq n$ and every $0<\varepsilon<1$ one has % there exist   constants $0<d_k<1<D_k$ such that
\begin{equation}\label{eq:dyD}
(1-\varepsilon) (|y_n^{(1)}|, \dots, |y_k^{(1)}|)(i)\leq (|y_n^{(k)}|, \dots, |y_k^{(k)}|)(i)\leq (1+\varepsilon)  (|y_n^{(1)}|, \dots, |y_k^{(1)}|)(i)
\end{equation}
for all but finitely many $i\in \mathbb N$.
\medskip

 We prove this statement by induction on $k$. For $k=1$ it is trivially true. The inductive step is as follows. Fix $\varepsilon < \tfrac{1}{2}$. Put $b=\|y_k^{(k)}\|_1$ and pick some $i$ for which \eqref{eq:dyD} holds so that for $k+1\leq m\leq n$ one has
 \begin{equation*}\label{eq:dyD*}
(1-\varepsilon)  |y_m^{(1)}(i) |\leq |y_m^{(k)}(i )| \leq (1+\varepsilon) | y_m^{(1)}(i)|.
\end{equation*}
Now, since $y_m^{(k+1)}=  y_m^{(k)}-\Phi_{m-k} y_k^{(k)}$ %, % for any $0<d_{k+1}<d_k$% and those indices $i$
, we have (the ``variable'' $i$ is omitted):
\begin{align*}
|y_m^{(k+1)}|
&\geq \left|y_m^{(k)}\right|-\left|\Phi_{m-k} y_k^{(k)}\right|\\
& \geq (1-\varepsilon)  |y_m^{(1)}|- \frac{|y_k^{(k)}|}{(m-k)!}\log^{m-k}\frac{b}{|y_k^{(k)}|}\\
&= (1-\varepsilon)  \frac{x}{m!}\log^{m}\frac{1}{x} -
\frac{|y_k^{(k)}|}{(m-k)!}\log^{m-k}\frac{b}{|y_k^{(k)}|}\\
&\geq (1-\varepsilon)  \frac{x}{m!}\log^{m}\frac{1}{x} -
\frac{ |y_k^{(1)}|}{(m-k)!} \log^{m-k}\frac{b}{ |y_k^{(1)}|}\\
&= (1-\varepsilon)  \frac{x}{m!}\log^{m}\frac{1}{x} -
\frac{ \frac{x}{k!}\log^{k}\frac{1}{x}}{(m-k)!} \log^{m-k}\frac{b}{ \frac{x}{k!}\log^{k}\frac{1}{x}}\\
&\geq (1-2\varepsilon)  \frac{x}{m!}\log^{m}\frac{1}{x}
\end{align*}
 for all but finitely many indices; see \eqref{eq:limfifj}.

The other inequality is easier: actually $|y_m^{(k+1)}(i )|\leq (1+\varepsilon) |y_m^{(k)}(i)|$ for $i$ large enough.
This ends the proof. Of course, the smaller $\varepsilon$
is, the more exceptional indices there will be. \end{proof}

For an arbitrary index set $I$, apply the same argument to the
couple $(\ell_{2/3}(I),\ell_2(I))$. Every vector involved has countable
support, and the proof reduces verbatim to the countable case. For real
spaces, just observe that complex spaces are real (aren't they?) and take the subspace of real families of $\ell(I, f_k)$.

\begin{remark*}
Regarding \eqref{eq:dyD}, it is worth noting that a pointwise inequality such as $(|x_n|,\dots, |x_1|)\leq
(|y_n|,\dots, |y_1|)$ does not imply any relation between the quasinorms of  $(x_n,\dots, x_1)$ and $(y_n,\dots, y_1)$ in $X_\theta^{(n)}$ unless $n=1$.
\end{remark*}

\section{Miscellaneous observations and counterexamples}

%In this section we will relax our notation by just calling Serre space to an element of the Serre closure of the class of Banach spaces.

In this section we will relax our notation (and ourselves) by simply calling an element of the Serre closure of the class of Banach spaces a Serre space.

\subsection{Serre spaces that are not iterated extensions of Banach spaces} A quasi-Banach space with trivial dual cannot be obtained as an iterated extension of Banach spaces; see you why? Actually, if $E$ has trivial dual and $Z$ is an iterated extension of Banach spaces then $L(E,Z)=0$,  The announced example is a quotient of  the continuous version of the space $X^{(2)}_{1/2}$ and could be constructed by interpolating the Lebesgue spaces $L_{2/3}$ and $L_2$ at $\theta=\tfrac{1}{2}$. Here, for $0<p<\infty$, we denote by $L_p$ the space of measurable functions $f:[0,1]\to \mathbb K$ for which the quasinorm $\|f\|_p=\left(\int_0^1|f(t)|^pdt\right)^{1/p}$ is finite. This time the ambient space is $L_{2/3}$ which is far from being locally convex and some complications arise. We rather prefer to construct a self-extension of $L_1$ directly as follows:
\begin{itemize}
    \item We fix any $p\in(0,1)$ (take $p=2/3$, now you are accustomed to it) and we consider $L_{2/3}$ as our containing space. The map $\Omega: L_1\to L_p$ given by
    $\Omega(f)= f\log (|f|/\|f\|_1)$ is quasilinear from $L_1$ to itself.
        \item Therefore $L_1\oplus_\Omega L_1=\{(g,f)\in L_p\times L_1: g-\Omega f\in L_1\}$ is a self-extension of $L_1$ with quasinorm $\|(g,f)\|_\Omega=\|g-\Omega f\|_1+\|f\|_1$. This is a kind of Kalton--Peck space; see \cite{kalton-peck} for the original construction (sequence spaces) and \cite{kalton-memoir} for its extension to function spaces.
        % \cite[Sections~10.8--10.9]{HMBST} for the quasilinear-map viewpoint.
    \end{itemize}

Let $D=\{f\in L_1: \Omega f\in L_1\}$. This space, usually called the domain of $\Omega$, is equipped with the quasinorm $\|f\|_D= \|\Omega f\|_1+\|f\|_1$, so that it is isometric to the subspace $\{(0,f)\in L_1\oplus_\Omega L_1\}$ of  $L_1\oplus_\Omega L_1$. The domain $D$ agrees, as a set, with the Orlicz function space associated with the convex function $t\mapsto t\log(e+t)$, and $\|\cdot\|_D$ is
equivalent to its Luxemburg norm, although we don't need this fact. This space has a considerable pedigree and is often called $L\log L$; see \cite[Section 6.3]{garling}.

%Indeed, take $f\ne0$, put
%$s=\|f\|_1$, $u=|f|/s$, and
%$$
%A=1+\int_0^1u\log_+u.
%$$
%Since $\int_{\{u\leq1\}}u|\log u|\leq1/e$, we have
%$$
%\|f\|_D=s\left(1+\int_0^1u|\log u|\right)\asymp sA.
%$$
%If $C$ is a sufficiently large absolute constant, then
%$$
%\int_0^1\Psi\left(\frac{u}{CA}\right)\leq 1,
%$$
%b%ecause $\log(e+v)\lesssim1+\log_+v$. Conversely, if
%$\int\Psi(u/t)\leq1$, then $t\geq1$ and
%\begin{align*}
%\int u\log_+u
%&\leq\log t+t\int\frac{u}{t}\log_+\frac{u}{t}\\
%&\leq\log t+t\leq2t.
%\end{align*}
%This proves the norm equivalence. In particular, $D$ is Banach.

\begin{example} The Serre space $R=(L_1\oplus_\Omega L_1)/D$ has trivial dual. In particular $R$ does not embed into an iterated extension of Banach spaces.\end{example}
\begin{proof}% The Serre assertion follows from being $D$ Banach and $L_1\oplus_\Omega L_1$ an extension of two Banach spaces.
We can identify $(L_1\oplus_\Omega L_1)/D$ with the so-called range space of $\Omega$, which consists of those functions $g$ such that $(g,f)\in L_1\oplus_\Omega L_1$ for some $f\in L_1$, quasinormed by
$\|g\|_R=\inf_{f\in L_1}\|g-\Omega f\|_1+\|f\|_1$.

Form the an exact sequence
$$
\xymatrix{
0\ar[r]&D \ar[r]^-\imath & L_1 \oplus_\Omega L_1 \ar[r]^-\pi& R\ar[r]&0}
$$
where $\imath(f)=(0,f)$ and $\pi(g,f)= g$. Given a measurable $A\subset[0,1]$ we denote the characteristic function of $A$ by $1_A$. Note that $\Omega(1_A)= 1_A\log(1/|A|)$, where $|A|$ is the Lebesgue measure of $A$. This can be used to estimate the quasinorm of these functions in $R$ as follows. For $0<|A|<1$ one has
$$
\|1_A\|_R=\inf_{f}\|1_A-\Omega f\|_1+\|f\|_1\leq
\inf_{\lambda}\|1_A-\Omega (\lambda 1_A)\|_1+\|\lambda 1_A\|_1.
$$
But since
$$
\|1_A-\Omega (\lambda 1_A)\|_1+\|\lambda 1_A\|_1= |A|\left(\left|1-\lambda\log\frac{1}{|A|}\right|+|\lambda|\right),
$$
taking $\lambda=1/\log(1/|A|)$ we obtain
$$
\|1_A\|_R\leq\frac{|A|}{\log(1/|A|)}.
$$
%The expression is not used at $|A|=1$: in the argument below we first
%partition $A$ into pieces of measure $|A|/n$.\medskip

Let $X$ be a Banach space and $T:R\to X$ an operator. Take any $A$ and for each $n$ let $A=A_1\cup \dots \cup A_n$ be a partition of $A$ into subsets of the same measure ($|A_i|=|A|/n$ for all $1\leq i\leq n$). Since $1_A= \sum_i 1_{A_i}$, one has
\[
\|T(1_A)\|_X
\leq \sum_i \|T(1_{A_i})\|_X
 \leq \|T\| \sum_{i\leq n} \|1_{A_i}\|_R\leq \|T\| \sum_{i\leq n} \frac{|A_i|}{\log(1/|A_i|)}=\|T\| \frac{|A|}{\log(n/|A|)}.
\]
%\begin{align*}
%\%|T(1_A)\|_X
%&\leq \sum_i \|T(1_{A_i})\|_X
% \leq \|T\| \sum_{i\leq n} \|1_{A_i}\|_R\\
%&\leq \|T\| \sum_{i\leq n} \frac{|A_i|}{\log(1/|A_i|)}\\
%&=\|T\| \frac{|A|}{\log(n/|A|)}.
%\end{align*}
Letting $n\to\infty$ we conclude that $T(1_A)=0$. Since it is relatively easy to check that these characteristic functions generate a dense subspace of $R$, we are done.\end{proof}

A few remarkable facts about this example:
\begin{itemize}
\item
Taking $A=[0,1]$ one sees that $c(n;R) \gtrsim \log n $ and since $R$ is a quotient of an extension of two Banach spaces one actually has $c(n;R)\sim \log n$.
\item

The space $R$ agrees with $L\log^{-1}L$, which is the  Orlicz function space associated to $t\mapsto t/\log(e+t)$.
\end{itemize}

Regarding the dual situation, we also have a radical counterexample:

\begin{example}\label{ex:lf1} %Let $k\geq 1$ an integer.
The space $\ell(f_1)$ is not a quotient of any iterated extension of Banach spaces.
\end{example}

The same can be said of $\ell(f_k)$ for any $k>0$ and actually we will show that if $E$ is an iterated extension of Banach spaces, then all operators $E\to \ell (f_k)$, with $k>0$, are compact. We have organized the proof avoiding all unnecessary theory. We consider $\ell(f_k)$ with the quasinorm provided by the first item of Corollary~\ref{cor:lorentz}. Clearly, the formal inclusion $J:\ell(f_k)\to \ell_1$ is bounded (by 1). This operator represents the so-called Banach envelope of $\ell(f_k)$ which means for every operator $u: \ell(f_k)\to X$ with $X$ Banach there is a unique operator $\tilde{u}:\ell_1\to X$ such that $u=\tilde{u}J$:
$$
\xymatrix{\ell(f_k) \ar[rd]_J \ar[rr]^u & & X\\
& \ell_1 \ar[ru]_{\tilde{u}}
}
$$
That is, $u$ extends to $\ell_1$. Ah, but we had promised not to theorize! Ok, the crucial point is:

\begin{lemma}\label{lem:compactness}
Let $Y \to Z  \to B$
be a short exact sequence of quasi-Banach spaces, with $B$ Banach and let $T:Z\to \ell (f_k)$ be an operator, $k>0$. If $JT|_Y$ is compact,
then $JT$ is compact.\end{lemma}

Is the just stated lemma the tip of the iceberg of a general result? It is very tempting to think that; however, as we shall shortly see, the proof relies on very specific properties of the space $\ell (f_k)$. Note, in particular, that for $x\in\ell (f_k)$ one has
\begin{equation}\label{eq:small-coordinates}
\|x\| \geq
\|x\|_1\left(
1+\log^k \frac{\|x\|_1}{\|x\|_\infty}
\right).
\end{equation}
Consequently, if $(x_n)_{n\geq 1 }$ is bounded in $\ell( f_k)$, then $\|x_n\|_\infty \to 0 \iff \|x_n\|_1 \to 0$. Keep this in mind.

As for $\ell_1$ what is needed for the proof is the following compactness criterion, probably due to M. Fr\'echet \cite[pp. 307--308]{frechet}: a bounded subset $K\subset\ell_1$ is relatively compact in the norm topology if and only if for every $\varepsilon>0$ there is a finite set of indices $F$ such that $\|(1-1_F) f\|_1\leq \varepsilon$ for every $f\in K$; see  \cite[Chapter I, ex.~6]{diestel}.

\begin{proof} Clearly, we may assume that $Y=\Ker q$, where  $q:Z\to B$ is the quotient map. Suppose that $JT$ is not compact.
Then there is a bounded sequence $(w_n)$ in $Z$ such that $\|Tw_n-Tw_k\|>1$ for $n\neq k$. Passing to a subsequence we may assume that $(Tw_n)$ is coordinatewise convergent. Letting $z_n=w_{2n}-w_{2n+1}$ we get a bounded sequence whose image under $T$ is coordinatewise null. Now, given any sequence of strictly positive scalars $(\varepsilon_n)$, we may choose a subsequence of $(z_n)$ that we don't relabel and a disjoint sequence of finitely supported $u_n\in \ell (f_1)$ such that $\|u_n-Tz_n\|<\varepsilon_n$, with $u_n=Tz_n$ on $F_n=\supp u_n$. In what follows, we proceed with the proof assuming that $u_n=Tz_n$. The necessary adjustments, including  the choice of $(\varepsilon_n)$, for the general case are left to the reader, but note that $\varepsilon_n=2^{-n}$ works. In any case we may assume that $\|z_n\|_E\geq \|u_n\|_1=1$ for every $n$. Fix an integer $m$ and put
\begin{equation*}
\overline{z}_j=\frac1m\sum_{(j-1)m< n\leq  jm}z_n,
\qquad
\overline{u}_j=\frac1m\sum_{(j-1)m< n\leq  jm}u_n.
\end{equation*}
Note that $\|\overline{u}_j\|_1=1$, but $\|\overline{u}_j\|_\infty\leq 1/m$, so $\|\overline{u}_j\|_E\geq \log m$ and $\|\overline{z}_j\|_Z\geq \|T\|^{-1} \log m$. Now the hypothesis that $B$ is Banach enters: we have $\|z_n\|_Z\leq M$ for some $M$ and all $n$, hence by the convexity of the norm
$$
\left\| q \overline{z}_j \right\|_B \leq M \|q\|;
$$
this bound is independent of $m$. Since $q$ is open there exists $\tilde{z}_j\in Z$ such that $q\tilde{z}_j=\overline{z}_j$ and %in such a way that
\begin{equation}\label{eq:Mprima}
\|\tilde{z}_j\|_Z\leq M'
\end{equation}
for all $j$, where $M'$ depends only on $M$ and $q$, but not on $m$. Consider now the sequence $y_j=\overline{z}_j-\tilde{z}_j$, which is bounded in $Z$ and contained in $Y$ --- we typically expect $\|y_j\|$ to be of size $\log m$ or so because of the summand $\overline{z}_j$. By the hypothesis on  $JT|_Y$, the points $JTy_j$ form  relatively compact subset of $\ell_1$. Let $\overline{F}_j= \supp \overline{u}_j$. These form a disjoint family and are all finite. Hence, by the compactness criterion,
$$
\| 1_{\overline{F}_j}Ty_j\|_1 = \| 1_{\overline{F}_j}T(\overline{z}_j-\tilde{z}_j)\|_1\To 0.
$$
But $\|\overline{u}_j - 1_{\overline{F}_j}T \overline{z}_j\|_1\leq  \|\overline{u}_j - T \overline{z}_j\|_1 \to 0
$ and, therefore,
$$
 \|\overline{u}_j - 1_{\overline{F}_j} T \tilde{z}_j\|_1 \To 0.
 $$

Look at the vector $1_{\overline{F}_j} T \tilde{z}_j$ for large $j$:
\begin{itemize}
\item  $\|1_{\overline{F}_j} T \tilde{z}_j\|_1 \sim \|\overline{u}_j \|_1$ is roughly 1.
\item  $\|1_{\overline{F}_j} T \tilde{z}_j\|_\infty \sim \|\overline{u}_j \|_\infty$ is at most $1/m$.
\item Hence $\|T\|\| \tilde{z}_j\| \geq \|1_{\overline{F}_j} T \tilde{z}_j\|$ is at least $\log m$, see \eqref{eq:small-coordinates}.
\end{itemize}
This is incompatible with the boundedness of the starting sequence $(z_n)$ in view of \eqref{eq:Mprima}.
\end{proof}

\begin{proof}[Verification of Example~\ref{ex:lf1}]
It is clear from Lemma~\ref{lem:compactness} that if $B$ is a Banach space and $T:B\to \ell (f_1)$ is bounded, then $JT: B\to \ell (f_1)\to \ell_1$ is compact --- just take $Z=0\oplus B$.

Let $Z$ be an iterated extension of Banach spaces. Then there are quasi-Banach spaces $0=Z_0\subset Z_1\subset\cdots\subset Z_r=Z$ such that $Z_{i}/Z_{i-1}$ is (isomorphic to) a Banach space for $0< i\le r$. Let $T:Z\to \ell (f_1)$ be  a bounded operator.
Applying Lemma~\ref{lem:compactness} to the exact sequences
$Z_{i-i}\to Z_{i}
\to Z_{i}/Z_{i-1}\to 0
$ shows successively that the restriction of $JT$ to each $Z_i$ is
compact. Thus $JT:Z\to \ell (f_1)\to \ell_1$ is compact and $T$ cannot be onto because $J: \ell (f_1)\to \ell_1$ is not compact. \end{proof}

Returning to Kanda's paper \cite{kanda}, we see that the containments $$\mathscr{Q}\mathscr{E}({\bf D}, {\bf E})\subset \mathscr{E}(\mathscr{Q}{\bf D}, \mathscr{Q}{\bf D}),\qquad \mathscr{S}\mathscr{E}({\bf D}, {\bf E})\subset \mathscr{E}(\mathscr{S}{\bf D}, \mathscr{S}{\bf E})$$ which hold in Abelian categories (\cite[Proposition 2.3, parts (4) and (5)]{kanda}), fail in the quasi-Abelian category of quasi-Banach spaces for ${\bf D}= {\bf E}= {\bf B}$, the class of Banach spaces. %Of course, the ``dual'' statement  $\mathscr{S}\mathscr{E}({\bf D}, {\bf E})\subset \mathscr{E}(\mathscr{S}{\bf D}, \mathscr{S}{\bf E})$ fails in the opposite category.
%To sum up: while the first three parts of \cite[Proposition 2.3]{kanda} remain valid in quasi-Abelian categories, all the other parts fail either in $\bf Q$ or in its opposite.

%\end{itemize}

\subsection{$p$-convex for all $0<p<1$ does not imply Serre}
In view of Lemma~\ref{lem:log->p}, it is natural to ask whether there exist spaces that are $p$-normable for every $p\in(0,1)$
but fail to be Serre. It is easy to construct one as follows: given a strictly increasing sequence $p_n$ converging to $1$, let $E_n$ be the space $\ell (f_n)$ equipped with an equivalent $p_n$-norm, which exists by the aforementioned lemma.
The space $\ell_1(\mathbb N, E_n)$, which consists in all families $(x_n)_{n}$ with $x_n\in E_n$ such that
$\sum_n \|x_n\|_{E_n}<\infty$, with the obvious quasinorm, is  $p$-normable for every $p\in(0,1)$
but fails to be Serre. The catch-22 here is that we have no control over the equivalence constants for the new quasinorms... and ultimately, it could happen that $\ell_1(\mathbb N, E_n)$
does not coincide with $\ell_1(\mathbb N, \ell f_n)$!

A more explicit example can be described as follows.
Let $p(n)$ be an increasing sequence in $(0,1)$, with $p(n)\to 1$ and let $k(n)$ be a strictly increasing sequence in $\mathbb N$. Put
\begin{equation}
   E= %E(p(\bullet),k(\bullet))
    \left\{x\in \mathbb R^{\mathbb N}:  \sum_{n\geq 1}  \left(\sum_{k(n)\leq i<k(n+1)} |x(i)|^{p(n)} \right)^{\frac{1}{p(n)}} < \infty \right\},
\end{equation}
with the obvious quasinorm. Note that $E$ is just an $\ell_1$-sum of finite-dimensional $\ell_p$-spaces. This space is $q$-convex for every $0<q<1$ because, once $q$ is fixed, we can split it as a direct sum of a finite-dimensional subspace (the sequences supported up to $\sum_{p(n)\leq q} k(n)$) and a $q$-normed subspace. Take $k(n)=2^n$ and let $a(n)\to\infty$ be any sequence such that $(a(n)/n)\log n\to 0$, for instance $a(n)=n^\alpha$ with $0<\alpha<1$. Then, letting
$$
\frac{1}{p(n)}= 1+ \frac{a(n)\log_2 n}{n},
$$
one easily sees that the sum of the $2^n$ unit vectors in the $n$-th chunk already yields $c(2^n;E)\geq n^{a(n)}$, while in a Serre space one must have $c(2^n)\leq n^k$ for some fixed $k$ and $n$ large.\medskip

%Talking about of spaces that are $p$-convex for every $p\in(0,1)$:

\subsection{On $p$-Banach spaces}
Characterizing $p$-convex spaces by the behaviour of the sequence $c(n;X)$
seems to be a rather challenging open problem, perhaps not very interesting. Observe that if $X$ carries a $p$-norm, one has
\begin{equation*}
    \left\|\sum_{i\leq n} x_i\right\|\leq \left(\sum_{i\leq n} \|x_i\|^p\right)^{1/p}\leq n^{1/p - 1} \left(\sum_{i\leq n} \|x_i\|\right),
\end{equation*}
so that $c(n;X)\leq n^{1/p - 1}$ for all $n$. Curiously enough, if $c(2;X)\leq 2^{1/p - 1}$ then $X$ can be given an equivalent $p$-norm, by the traditional proofs of the Aoki--Rolewicz theorem. Of course, if $X$ is has an equivalent $p$-norm, then $c(n;X)\lesssim n^{1/p - 1}$
and the relevant question is whether the converse is true (which is equivalent to a problem of Peck \cite[Remark 1 on p. 133]{peck}).
On the other hand, in his earliest study on the three-space problem, Kalton introduced the sequence
$$
a(n;X)=\sup_{\|x_i\|\leq 1}  \left\|\sum_{i\leq n} x_i\right\|,
$$
proved that in a $p$-convex space one always has  $a(n;X)\lesssim n^{1/p}$, and asked whether the converse is true; see \cite{kalton-locally}, Proposition 2.1 and the ensuing paragraph, as well as the remark following his Proposition 2.2. Clearly, $a(n;X)\leq n\,c(n;X)$, so $c(n;X)\lesssim n^{1/p - 1}$ implies $a(n;X)\lesssim n^{1/p}$ and, thus, an affirmative answer to Kalton's question would have made us all happy.\smallskip

 %would solve Problem~\ref{prob:cn->p?} as well.

When we typed this prompt into an AI interface (namely, GPT-6 Astra), the algorithm spat out that a quasi-Banach space $X$ has $c(n; X)\lesssim n^{1/p - 1}$ if and only if $a(n;X)\lesssim n^{1/p}$ and that this does not imply $p$-convexity, the counter-example being a suitable modification of a celebrated construction of Tzafriri \cite{tz}, which in turn is a suitable modification of Tsirelson space.

At this moment, we know of no characterization of $p$-convexity in terms of the behaviour of $c(n;X)$, which is not so surprising because  $c(n;X)$ is designed to detect (deviation from) local convexity; see the next section for the correct constants.
\medskip

An important result by Kalton states that every extension of two $p$-Banach spaces, $p\in(0,1]$, is $q$-convex for every $0<q<p$. Therefore, ${\bf B}(p^-)$, the class of those quasi-Banach spaces that admit an equivalent $q$-norm for every $q<p$, is a Serre class. One has:

\begin{proposition}[Kalton] Let $X$ be a quasi-Banach space and $0<q<1$. The following conditions are equivalent:
\begin{itemize}
    \item[{(a)}] $X$ is $p$-convex for some $p>q$.
     \item[{(b)}] $c(n;X)/n^{1/q -1}\to 0$ as $n\to\infty$. % for every  $\gamma>0$;
      \item[{(c)}] $a(n;X)/n^{1/q}\to 0$ as $n\to\infty$.
\end{itemize}
\end{proposition}
\begin{proof} The implications \textup{(a)}$\implies$\textup{(b)}$\implies$\textup{(c)} have been proved. The equivalence \textup{(a)}$\iff$\textup{(c)} is contained in \cite[Proposition~2.2]{kalton-locally}.
\end{proof}

The immediate consequence is the following characterization of ${\bf B}(1^-)$:

\begin{corollary} For a quasi-Banach space $X$ the following conditions are equivalent:
\begin{itemize}
    \item[{(a)}] $X$ is $p$-convex for every $p\in(0,1)$;
     \item[{(b)}] $c(n;X)\lesssim n^{\gamma}$ for every  $\gamma>0$;
      \item[{(c)}] $a(n;X)\lesssim n^{\alpha}$ for every  $\alpha>1$.
\end{itemize}
\end{corollary} \begin{proof} Assume \textup{(a)} and fix $\gamma>0$. Choose $p<1$ so close to $1$
that $1/p-1<\gamma$. A $p$-convex quasinorm gives
$c(n;X)\lesssim n^{1/p-1}\leq n^\gamma$, which proves \textup{(b)}.
The implication \textup{(b)}$\implies$\textup{(c)} follows at once
from $a(n;X)\leq n c(n;X)$. Finally, assume \textup{(c)} and fix $0<q<1$. Choose
$1<\alpha<1/q$. Then
$$
\frac{a(n;X)}{n^{1/q}}\lesssim n^{\alpha-1/q}\longrightarrow0.
$$
Kalton's proposition yields $p$-convexity for some $p>q$, which in turn
implies $q$-convexity. Since $q<1$ was arbitrary, \textup{(a)} follows.
\end{proof}

\subsection{The Serre closure of $p$-Banach spaces, $0<p<1$}The same ideas can be used to characterize the Serre closure of $p$-Banach spaces in $\bf Q$ provided  we do not insist on doing so literally. Fix $p\in(0,1)$. Given a quasi-Banach space $X$ define for $n\geq 2$$$ c_p(n;X)=\sup_{x_i\in X}\frac{\|\sum_{i\leq n}x_i\|^p}{\sum_{i\leq n}\|x_i\|^p}.$$
Note that $X$ is $p$-convex if and only if $c_p(n;X)$ is bounded. Kalton and Peck proved in \cite{kalton-peck} that the class of $p$-convex spaces fails to be closed under extensions. One has the following extension of our main theorem:

\begin{theorem*} Let $0<p\leq 1$. For a quasi-Banach space $X$ the following are equivalent:
\begin{itemize}
\item $X$ belongs to the Serre closure of $p$-Banach spaces in $\bf Q$.
\item $c_p(n;X) \lesssim \log^k n$ for some $k\in \mathbb N$.% for every $n\geq2$.
\item $X$ is a
quotient of $\ell(I,f_l^p)$ for some set $I$ and some $l\in\mathbb N$.
\end{itemize}
\end{theorem*}

\subsection{On the Serre closure of Hilbert spaces}

What about Hilbert spaces? First of all, Hilbert spaces are $B$-convex. A quasi-Banach space $X$ is called $B$-convex if
$b(n;X)/n\to 0$, where
$$
b(n;X) = \sup_{\|x_i\|\leq 1} \min_{\varepsilon_i=\pm1}\left\|\sum_i \varepsilon_i x_i\right\|.
$$
Kalton proved in \cite[Theorem 2.6]{kalton-locally} that $B$-convexity is a
three-space property and that every $B$-convex quasi-Banach space is locally
convex. Therefore, if $\bf H$ denotes the class of Hilbert spaces, then
$\langle\bf H\rangle_{\bf Q}=\langle\bf H\rangle_{\bf B}$. On the other hand, the dual of a Hilbert space is again a Hilbert space; it follows that $\langle\bf H\rangle_{\bf B}$ is a self-dual class and, actually, each  ${\bf H}^{(k)}$ is --- the notation should be obvious.

Characterising the Banach spaces in the Serre closure of $\bf H$ is surely a difficult task. A reasonable try is recalling the definition of $n$-th type 2 constant: If $X$ is a Banach space, $a_2(n,X)$ is the smallest positive number $a$ such that
$$\underset{\varepsilon_i=\pm 1} {\mathrm{ Average}}\quad  \Big\|\sum_{i=1}^n \varepsilon_i x_i\Big\|^2 \leq a^2 \sum_{i=1}^n \|x_i\|^2
$$
for all $x_1, \dots, x_n \in X$. It is shown in \cite[Lemma 4.3]{ccc} that if $X$ belongs to ${\bf H}^{(k)}$, then $a_2(n;X)\lesssim \log^{k-1} n$  and that this is optimal because for every $k>1$ there is an ($k$ times) iterated extension of $\ell_2$ with $a_2(n)\sim\log^{k-1} n$; see \cite[Lemmata 4.3 and 4.4]{ccc}. Since type constants can only decrease when passing to subspaces or quotients, $a_2(n;Z)\lesssim \log^{2^k -1}(n)$ for every $Z$ in $(\mathscr{S\hspace{-3pt}QE})^k\bf H$. On the other hand, if one fixes $1<p<2$ and then takes $q$ to be the conjugate exponent (meaning that $1=1/p+1/q$; for instance, one may take $p=3/2$, so that $q=3$), then interpolating between $\ell_p$ and $\ell_q$ by the complex method gives $X_{1/2}=\ell_2$, % at $\theta=1/2$
 and proceeding as in Proposition~\ref{prop:lfk is serre} one can show that the Orlicz space $\ell(f_k^2)$ is a closed subspace of an iterated extension of $k+1$ copies of $\ell_2$. This result is stated in   \cite[Proposition 5.2]{cgp}, but the proof is incorrect. Actually, Proposition~\ref{prop:lfk is serre} is clearly inspired by \cite[Proposition 5.2]{cgp}; as a fair trade, our proof can be easily adapted to provide a correct proof of the latter.
 \medskip

However, those Orlicz spaces $\ell(f_k^2)$ fail to generate $\langle \bf H\rangle_{\bf B}$ because they all have cotype 2 (given $\alpha\geq 0$, the Orlicz  sequence space associated with the function $t^2\log^\alpha(1/t)$ has cotype $2$ and type $2-\varepsilon$ for each $\varepsilon>0$, see \cite[Corollary~13]{katir}), and so do their subspaces and quotients. On the other hand, if one has a short exact sequence $K \to Z \to H$ in which $H$ and $K$ are Hilbert spaces, then $Z$ has type or cotype 2 if and only if the sequence splits, which is equivalent to $Z=K\oplus H$ being itself a Hilbert space. The moral is that a nontrivial twisted Hilbert space $Z$ is neither a subspace nor a quotient of a space $L$ of either type 2 or cotype $2$. One could also consider the dual spaces of those $\ell(f_k^2)$. Incidentally, $\ell(f_k^2)^*$ is isomorphic to the sequence space associated to $t^2\log^{-2k}(1/t)$; see \cite[Chapter 4, Section c]{Lind-Tzaf} for an exhaustive analysis of the sequence spaces generated by the Orlicz functions $t^p\log^\alpha(1/t)$ for $p>1; \alpha>0$. By the type / cotype duality, $\ell(f_k^2)^*$ has type $2$ and cotype $2+\varepsilon$ for each $\varepsilon>0$. But even that $\ell(f_1^2)\oplus \ell(f_1^2)^*$ (say) admits a nontrivial twisted Hilbert space as a subspace or quotient is not enough, as we show now. Let $Z_2$ be the Kalton--Peck space, the most famous and nontrivial twisted Hilbert space ever, whose autobiography up to 2023 can be found in \cite[Section 10.9]{HMBST}:

\begin{claim*} If $C$ is of cotype 2 and has nontrivial type, and $T$ is of type 2, then  $Z_2$ is neither a subspace nor a quotient of $C\oplus T$.\end{claim*}

\begin{proof} Every operator defined on $Z_2$ is either strictly singular or an isomorphism on a copy of $Z_2$;
see \cite[Proposition 10.9.1]{HMBST}.
Let $u: Z_2\to C\oplus T$ be an operator. Since neither $C$ nor
$T$ contains copies of $Z_2$, the coordinate maps $\pi_1u$ and $\pi_2u$ are
strictly singular. Strictly
singular operators form an operator ideal, so $u$ is stricty singular. In particular, $u$ cannot be an embedding. The quotient part follows from that taking into account that $Z_2$ is isomorphic to its dual. \end{proof}

In particular, $Z_2$  is neither a subspace nor a quotient of any finite direct sum of the Orlicz sequence spaces generated by the functions
$t^2\log^\alpha(1/t)$ with real $\alpha$. A reasonable guess such as: ``If $X$ is a Banach space such that $a_2(n, X\oplus X^*)\lesssim \log^k n$ for some $k$, then $X\in \langle\bf H\rangle_{\bf B}$'' seems doomed to fail like a rook pawn: weak Hilbert spaces \cite[Chapter 12]{pisier2}, whose definition is not important now, exist and satisfy the estimate
$a_2(n, X\oplus X^*) \lesssim \log^k n$
% BEGIN CODEX CHANGE C37 (removed an overclaim about subquotients)
for all $k>0$. Besides, there exist weak Hilbert spaces \cite[Remark~(i), p.~213]{pisier2} that have no subspace or quotient isomorphic to $\ell_2$:
the $2$-convexified Tsirelson space is one; see \cite[Notes and Remarks, p.~215]{pisier2}.
%that do not have subspaces or quotients contain isomorphic copies of $\ell_2$. In particular, examples such as Tsirelson's $2$-convexified space $\mathcal T_2$ are weak Hilbert and contain no subspace or quotient isomorphic to $\ell_2$ \cite[Notes and Remarks, p.~215]{pisier2}.
 If $Z$ is an iterated extension of Hilbert spaces, the $\ell_2$-saturation of $Z$ excludes $T_2$ as a subspace; and $\ell_2$-saturation of $Z^*$ excludes $T_2$ as a quotient of $Z$. This does not control subspaces of quotients or quotients of subspaces, and
therefore one might still have $T_2\in \langle\bf H\rangle_{\bf B}$. Does every quotient of an iterated extension of Hilbert spaces contain $\ell_2$? In general, quotients of $\ell_2$-saturated spaces need not contain $\ell_2$: Argyros and Raikoftsalis proved in \cite{at} that every separable reflexive Banach space is a quotient of a reflexive complementably $\ell_p$-saturated space for any $1<p<\infty$.\smallskip

So, against all odds, the rook pawn might yet queen.
%It is natural to ask whether every quotient of an iterated twisted Hilbert space contains $\ell_2$ (or is $\ell_2$-saturated). In general, quotients of $\ell_2$-saturated spaces need not contain $\ell_2$: Argyros and Raikoftsalis proved in \cite{at} that every separable reflexive Banach space is a quotient of a reflexive complementably $\ell_p$-saturated space for any $1<p<\infty$.

\end{document}